\documentclass{amsart}
\usepackage[utf8]{inputenc}
\usepackage{amsmath}
\usepackage{amsthm}
\usepackage{amssymb}
\usepackage{enumitem}
\usepackage{hyperref}
\usepackage{pdfsync}
\usepackage{footmisc}
\usepackage{color}
\usepackage{tikz,pgf}
\usetikzlibrary{shapes,decorations,arrows,calc,arrows.meta,fit,positioning}
\usetikzlibrary{arrows, automata}
\tikzset{
    -Latex,auto,node distance =1 cm and 1 cm,semithick,
    state/.style ={ellipse, draw, minimum width = 0.7 cm},
    point/.style = {circle, draw, inner sep=0.04cm,fill,node contents={}},
    bidirected/.style={Latex-Latex,dashed},
    el/.style = {inner sep=2pt, align=left, sloped}
}

\newenvironment{pict}[2]
	{\setlength{\unitlength}{1mm}
	\begin{center}
	\begin{picture}(#1,#2)
	\scriptsize
}
	{\end{picture}
	\end{center}
 	\noindent}

\usepackage{amsgen}
\usepackage{amstext}
\usepackage{amsbsy}
\usepackage{amsopn}
\usepackage{amsfonts}
\usepackage[mathscr]{euscript}

\newcommand{\D}{\overline{\Delta}}
\newcommand{\A}{\mathcal{A}^{\pm 1}}

\begin{document}

\newtheorem{theorem}{Theorem}[section]
\newtheorem{lemma}[theorem]{Lemma}
\newtheorem{corollary}[theorem]{Corollary}
\theoremstyle{definition}
\newtheorem{defn}[theorem]{Definition}
\newtheorem{question}[theorem]{Question}
\newtheorem{prop}[theorem]{Proposition}
\newtheorem{remark}[theorem]{Remark}
\newtheorem*{conj}{Conjecture}
\newtheorem*{thm*}{Theorem}

\author[M. Casals-Ruiz]{Montserrat Casals-Ruiz}
\address{Ikerbasque - Basque Foundation for Science and Matematika Saila,  UPV/EHU,  Sarriena s/n, 48940, Leioa - Bizkaia, Spain}
\email{montsecasals@gmail.com}

\author[I. Kazachkov]{Ilya Kazachkov}
\address{Ikerbasque - Basque Foundation for Science and Matematika Saila,  UPV/EHU,  Sarriena s/n, 48940, Leioa - Bizkaia, Spain}
\email{ilya.kazachkov@gmail.com}

\author[A. Zakharov]{Alexander Zakharov}
\address{Chebyshev Laboratory, St Petersburg State University, 14th Line 29B, Vasilyevsky Island, St.Petersburg, 199178, Russia, and The Russian Foreign Trade Academy, 4a Pudovkina street, 119285, Moscow, Russia.}
\email{zakhar.sasha@gmail.com}

\title{Commensurability of Baumslag-Solitar groups}

\maketitle
\begin{abstract}
    In this paper we classify Baumslag-Solitar groups up to commensurability. In order to prove our main result we give a solution to the isomorphism problem for a subclass of Generalised Baumslag-Solitar groups.
\end{abstract}

\section{Introduction}
The central objects of this paper are the \emph{Generalised Baumslag-Solitar groups}, or GBS groups, for short. A GBS group is simply the fundamental group of a finite graph of groups in which all vertex and edge groups are infinite cyclic. As the name suggests, GBS groups are a natural generalisation of \emph{Baumslag-Solitar groups}, which were introduced in \cite{BS} as examples of non-Hopfian groups. By definition, a Baumslag-Solitar group is an HNN-extension whose base group and associated subgroups are infinite cyclic, that is the fundamental group of a graph of groups with only one vertex and one edge and so it has a presentation of the form $BS(m,n)=\langle a,t \mid t^{-1}a^m t =a^n \rangle$, $m,n \in \mathbb Z\setminus \{0\}$. 

GBS groups have appeared in the study of finitely generated groups of cohomological dimension $2$, see \cite{K1}, splittings of groups and JSJ decompositions, see \cite{K2}, the study of one-relator groups, see \cite{C,Mc, Pi} and mapping tori, see \cite{L2}.

Two groups are called {\it (abstractly) commensurable} if they have
isomorphic subgroups of finite index. This is an equivalence relation, and finitely generated commensurable groups are quasi-isometric.
Gromov suggested to study groups from geometric point of view and understand the relation between these two concepts, as well as study quasi-isometry and commensurability classification of groups. The classification of groups up to commensurability has a long history and a number of solutions for very diverse classes of groups, see, for instance, \cite{CKZ1} and references there.

(G)BS groups have been studied both from geometric and algebraic perspectives. Geometrically, there is a number of strong results classifying the class of GBS groups up to quasi-isometry. In \cite{FM}, Farb and Mosher studied the problem for solvable Baumslag-Solitar groups and established that solvable Baumslag-Solitar groups are quasi-isometrically rigid, that is $BS(1,m)$ and $BS(1,n)$ are quasi-isometric if and only if they are commensurable if and only if there exist $r, k$ and $l$ such that $n=r^k$ and $m=r^l$. 

In contrast, Whyte showed that the non-solvable Baumslag-Solitar groups and most GBS groups are not at all quasi-isometrically rigid. In particular, if $n\ne \pm m$, $n,m\ne \pm 1$, then $BS(m,n)$ is quasi-isometric to $BS(2,3)$, see \cite{W}. On the other hand, the unimodular Baumslag-Solitar groups, that is $BS(\pm m,\pm m)$, $m\in \mathbb N$, $m > 1$ are commensurable, and so quasi-isometric, to $F_2 \times \mathbb Z$. At the same time, Whyte showed that $BS(m,n)$ and $BS(p,q)$ are not commensurable, whenever $m$ and $n$ are coprime, and $p$ and $q$ are coprime.

The progress achieved in understanding geometric aspects of GBS groups is in sharp contrast with the algebraic side: the isomorphism problem for GBS groups and their classification up to commensurability are still open. The isomorphism problem has been studied since the very introduction of GBS groups in the 90's, and the problem of classification of Baumslag-Solitar groups up to commensurability since at least the late 90's, alongside their quasi-isometric classification, and was formally formulated by Levitt in \cite{L5}.

There are, however, some partial results on the isomorphism problem and the commensurability classification. In a series of papers \cite{F1,F2,F3,CF1, CF2}, Clay and Forester studied GBS groups via their actions on the Bass-Serre trees and the structure of their splittings. More precisely, the authors showed that two splittings of the groups are in the same deformation space and they are related by a sequence of moves described in \cite{CF2}. In general, the number of reduced graphs in the deformation space may be infinite. In the special cases when there are only finitely many reduced graphs in the deformation space, one can obtain useful information about $Out(G)$, see \cite{GHMR,L3,P} and solve the isomorphism problem for some subclasses: when $Out(G)$ does not contain a non-abelian free group, see \cite{L3}, when the modular groups contain no integers other than $\pm 1$, see \cite{F2}, for GBS groups whose labeled graphs have first Betti number at most one, see \cite{CF1} and for GBS groups where one of the underlying graphs has a sole mobile edge, see \cite{D2}. 

First results on commensurability of (generalised) Baumslag-Solitar groups are closely linked to the study of their quasi-isometric classification, see \cite{FM,W}.  In \cite{W} (see also \cite{D1}), Whyte described finite index subgroups of $BS(m,n)$, where $\gcd(m,n) = 1$ and showed that no two groups in this class  are  commensurable. In  his  work,  see  \cite{L2},  Levitt  studied  the  class  of  GBS groups that generalises the condition $\gcd(m,n) = 1$ for $BS(m,n)$, that is the class of GBS groups without proper plateaus. In his work, Levitt showed that finite index subgroups of such GBS groups correspond to covers, described their rank and classified some groups from the class up to commensurability, see \cite[Lemma 6.4, Proposition 6.5]{L2}.

\bigskip

In this paper we address two central algebraic questions for (generalised) Baumslag-Solitar groups: the isomorphism problem for a class of GBS groups and the classification of Baumslag-Solitar groups up to commensurability.

Our first result is an explicit solution to the isomorphism problem for a certain class of GBS groups whose deformation space contains infinitely many reduced graphs. As we noticed, the previous results on the isomorphism problem were for classes of GBS groups with finitely many reduced graphs in the deformation space. More precisely, we solve the isomorphism problem for finite index subgroups of groups $G_{1,q}^d$ which are fundamental groups of the bouquets of Baumslag-Solitar groups $BS(1,q)$, see Equation \eqref{eq:gdpq} in the case $p=1$. 

\begin{theorem}\label{isothm}
The isomorphism problem is decidable for finite index subgroups of groups $G_{1,q}^d$, $q,d \in \mathbb N$, i.e. there is an algorithm that given two finite index subgroups  $H_1 \leq G_{1,q_1}^{d_1}$, $H_2 \leq G_{1,q_2}^{d_2}$ (given by their finite generating sets) decides whether or not $H_1$ and $H_2$ are isomorphic.
\end{theorem}

In order to do so, we describe a normal form for finite index subgroups of groups $G_{1,q}^d$. It is our belief that the ideas to approach the isomorphism problem for GBS groups with mobile edges can be extended to more general classes and may be an important ingredient together with the other established techniques for a solution of the isomorphism problem for GBS groups in general.

We then use the structural results on finite index subgroups of GBS groups $G_{1,q}^d$ to give a complete classification of Baumslag-Solitar groups up to commensurability. In contrast with the complete lack of geometric rigidity, we show that non-solvable Baumslag-Solitar groups are algebraically rigid: apart from some exceptions listed below, they are non-commensurable. More precisely, we prove the following:

\begin{theorem}\label{thm1}
    Let $G_1=BS(m_1,n_1)$ and $G_2=BS(m_2,n_2)$ be two Baumslag-Solitar groups, where $1 \le |m_i| \le n_i$, $i=1,2$. 
    Then the groups $G_1$ and $G_2$ are commensurable if and only if one of the following holds:
    \begin{enumerate}
        \item $|m_1|=|m_2|=1$ and $n_1, n_2$ are powers of the same integer, i.e. 
        $$
        BS(1, n^{k_1}) \sim BS(1, n^{k_2}), n, k_i \in \mathbb N;
        $$
        \item $n_1=n_2$ and $m_1= \pm m_2$, i.e.
        $$
        BS(m_1,n_1) \sim BS(\pm m_1,n_1);
        $$
        \item $|m_1|>1$, $|m_2| >1$, $m_1  \mid  n_1$, $m_2 \mid n_2$ and $\frac{n_1}{|m_1|}=\frac{n_2}{|m_2|}$, i.e.
        $$
        BS(\pm k,kn) \sim BS(\pm l,ln), k,l,n \in \mathbb N, k,l >1.
        $$
    \end{enumerate}
    Here $\sim$ denotes commensurability relation.
\end{theorem}

\paragraph{\emph{Acknowledgement.}} This work was supported by ERC grant PCG-336983, Basque Government Grant IT974-16 and Spanish Government grant MTM2017-86802-P. The third author was partially supported by the grant 346300 for IMPAN from the Simons Foundation and the matching 2015-2019 Polish MNiSW fund. The authors would like to thank G. Levitt for discussions and encouragement.

\section{Preliminaries}

We assume that the reader is familiar with Bass-Serre theory, see \cite{Serre}.

A \emph{Generalised Baumslag-Solitar group}, or GBS group, for short, is simply the fundamental group of a finite graph of groups in which all vertex and edge groups are infinite cyclic. A GBS group is completely defined by a finite directed graph $\Gamma$ and a labelling: for each oriented edge $e$ there is a label $A (e) \in \mathbb{Z} \setminus \{0 \}$ defining the embedding at the origin of $e$ and $\Omega(e) \in \mathbb{Z} \setminus \{0 \}$ defining the embedding at the end of $e$. 

A graph of groups is called \emph{reduced} if for every edge which is not a loop none of the two embeddings of the edge group into the corresponding vertex groups is an isomorphism. For GBS groups, this just means that if $e$ is not a loop then $A(e) \neq \pm{1}$ and $\Omega(e) \neq \pm{1}$. A \emph{simple cycle} (in a graph) is a closed path  without repetitions of vertices (except for the first and last vertices in the path) or edges. In a directed graph, a \emph{directed simple cycle} is a simple cycle with consistent edge orientation.

By convention, we will always include the inverses of oriented edges in our graph. The original directed edges of the graph will be called \emph{positive} edges, and their inverses, \emph{negative}. 

For an oriented edge $e$ of a directed graph we denote by $\alpha(e)$ its initial vertex, by $\omega(e)$ its terminal vertex and by $e^{-1}$ its inverse. For a GBS graph we have $A(e)=\Omega(e^{-1})$.

By definition, every GBS group can be studied via its actions on Bass-Serre trees with infinite cyclic vertex and edge stabilizers. Note that for GBS groups the graph of groups splitting is determined uniquely by the action on a Bass-Serre tree (and vice versa, which is always true). The indexed graph is then almost uniquely determined, up to two allowed change of sign moves: changing the signs of both labels of a given edge, and changing the signs of all labels of edges at some given vertex, see \cite{F2}.

By convention, we establish that $a^g=gag^{-1}$. By $gcd(p,q)$ we denote the greatest common divisor of $p$ and $q$.

The theory of {\it deformation spaces} is crucial for the study of GBS groups. We now recall some of the facts which will be used and we refer the reader to  \cite{F1}, \cite{F2} and \cite{GL} for definitions and basic properties. For any GBS group, which is not isomorphic to $\mathbb{Z} \times \mathbb{Z}$ or the Klein bottle group (i.e., to $BS(1,1)$ or $BS(1,-1)$), the set of elliptic subgroups is the same for all its actions on Bass-Serre trees with infinite cyclic vertex and edge stabilizers, see \cite{F1}. Therefore, any two splittings of such a GBS group with infinite cyclic vertex and edge groups are in the same deformation space. This means that these graphs of groups are related via a sequence of collapse and expansion moves (also called elementary deformations), which are defined below.

Any two reduced graphs of groups in the same deformation space are also related by a sequence of 3 types of moves: slide moves, induction moves and $\mathcal{A}^{\pm 1}$-moves. We now illustrate these moves in the GBS groups case; for more details, for the case of arbitrary deformation spaces and proofs see \cite{CF1}, \cite{CF2}.

Collapse and expansion moves are as follows:
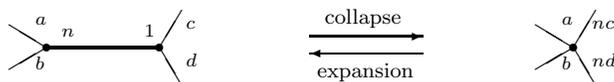
\begin{pict}{90}{10}
\thicklines
\put(10,5){\circle*{1}}
\put(25,5){\circle*{1}}
\put(10,5){\line(1,0){15}}

\thinlines
\put(45,6.5){\vector(1,0){15}}
\put(60,4.2){\vector(-1,0){15}}
\put(47,8.2){\footnotesize collapse}
\put(46,1){\footnotesize expansion}
\put(25,5){\line(3,5){3}}
\put(25,5){\line(3,-5){3}}

\put(10,5){\line(-5,3){5}}
\put(10,5){\line(-5,-3){5}}

\put(80,5){\circle*{1}}
\put(80,5){\line(-5,3){5}}
\put(80,5){\line(-5,-3){5}}

\put(80,5){\line(3,5){3}}
\put(80,5){\line(3,-5){3}}

\put(8.5,8){$a$}
\put(8.5,2){$b$}
\put(12,6.5){$n$}
\put(23,6.5){$1$}
\put(28.5,7.5){$c$}
\put(28.5,2.5){$d$}

\put(78.5,8){$a$}
\put(78.5,2){$b$}
\put(82.5,7.5){$nc$}
\put(82.5,2.5){$nd$}
\end{pict}

There are two slide moves:
\smallskip
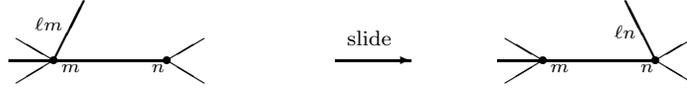
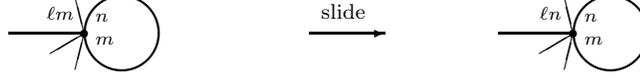
\begin{pict}{100}{11}
\thicklines
\put(75,3){\circle*{1}}
\put(90,3){\circle*{1}}
\put(75,3){\line(1,0){15}}
\put(90,3){\line(-1,2){4}}

\thinlines
\put(47.5,3){\vector(1,0){10}}
\put(49,5){\footnotesize slide}
\put(90,3){\line(5,3){5}}
\put(90,3){\line(5,-3){5}}
\put(75,3){\line(-5,3){5}}
\put(69,3){\line(1,0){6}}
\put(75,3){\line(-5,-3){5}}

\put(76,1.5){$m$}
\put(88,1.5){$n$}
\put(84.5,6){$\ell n$}

\thicklines
\put(10,3){\circle*{1}}
\put(25,3){\circle*{1}}
\put(10,3){\line(1,0){15}}
\put(10,3){\line(1,2){4}}

\thinlines
\put(25,3){\line(5,3){5}}
\put(25,3){\line(5,-3){5}}
\put(10,3){\line(-5,3){5}}
\put(4,3){\line(1,0){6}}
\put(10,3){\line(-5,-3){5}}

\put(11,1.5){$m$}
\put(23,1.5){$n$}
\put(7.5,7){$\ell m$}
\end{pict} 
and

\begin{pict}{100}{10}
\thicklines
\put(82.5,5){\circle*{1}}
\put(87.5,5){\circle{10}}
\put(72.5,5){\line(1,0){10}}

\thinlines
\put(47.5,5){\vector(1,0){10}}
\put(49,7){{\footnotesize slide}}
\put(82.5,5){\line(-5,-3){4.5}}
\put(82.5,5){\line(-1,-4){1.2}}
\put(82.5,5){\line(-1,4){1.2}}

\put(84,3.5){$m$}
\put(84,6.5){$n$}
\put(78,6.8){$\ell n$}

\thicklines
\put(17.5,5){\circle*{1}}
\put(22.5,5){\circle{10}}
\put(7.5,5){\line(1,0){10}}

\thinlines
\put(17.5,5){\line(-5,-3){4.5}}
\put(17.5,5){\line(-1,-4){1.2}}
\put(17.5,5){\line(-1,4){1.2}}

\put(19,3.5){$m$}
\put(19,6.5){$n$}
\put(12.5,6.8){$\ell m$}
\end{pict} 
An induction move is as follows:

\smallskip
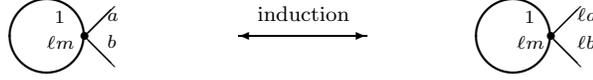
\begin{pict}{80}{10}
\thicklines
\put(6,5){\circle{10}}
\put(11,5){\circle*{1}}
\put(68,5){\circle{10}}
\put(73,5){\circle*{1}}

\thinlines
\put(11,5){\line(1,1){4}}
\put(11,5){\line(1,-1){4}}
\put(73,5){\line(1,1){4}}
\put(73,5){\line(1,-1){4}}

\put(14,7){$a$}
\put(14,3.5){$b$}
\put(7,6.7){$1$}
\put(6,3.3){$\ell m$}
\put(31.5,5){\vector(1,0){17}}
\put(48.5,5){\vector(-1,0){17}}
\put(34,7){\footnotesize induction}
\put(76.5,7){$\ell a$}
\put(76.5,3.3){$\ell b$}
\put(69.5,6.7){$1$}
\put(68.5,3.3){$\ell m$}
\end{pict}
Both directions of the move are considered induction moves. This move decomposes into a sequence of  elementary deformations as follows: 
\smallskip
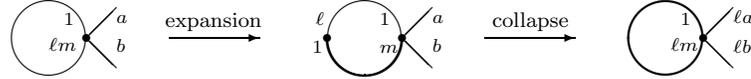
\begin{pict}{100}{10}
\thicklines
\put(48,5){\oval(10,10)[b]}

\put(88,5){\circle{10}}
\put(93,5){\circle*{1}}

\thinlines
\put(6,5){\circle{10}}
\put(11,5){\circle*{1}}

\put(48,5){\circle{10}}
\put(43,5){\circle*{1}}
\put(53,5){\circle*{1}}

\put(11,5){\line(1,1){4}}
\put(11,5){\line(1,-1){4}}

\put(53,5){\line(1,1){4}}
\put(53,5){\line(1,-1){4}}

\put(93,5){\line(1,1){4}}
\put(93,5){\line(1,-1){4}}

\put(15,7){$a$}
\put(15,3.2){$b$}
\put(8,6.7){$1$}
\put(6,3.3){$\ell m$}

\put(22,5){\vector(1,0){12}}
\put(21.5,6.5){\footnotesize {expansion }}

\put(57,7){$a$}
\put(57,3.2){$b$}
\put(50,6.7){$1$}
\put(50,3){$m$}
\put(41.5,6.7){$\ell$}
\put(41,3){$1$}

\put(64,5){\vector(1,0){12}}
\put(65,6.5){\footnotesize collapse}

\put(97,7){$\ell a$}
\put(97,3){$\ell b$}
\put(90,6.7){$1$}
\put(88.7,3.3){$\ell m$}
\end{pict}

Finally, the $\mathcal{A}^{\pm 1}$-moves are below:
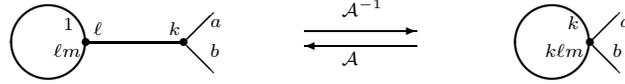
\begin{pict}{90}{10}
\thicklines
\put(6,5){\circle{10}}
\put(11,5){\circle*{1}}
\put(73,5){\circle{10}}
\put(78,5){\circle*{1}}
\put(24,5){\circle*{1}}
\put(11,5){\line(1,0){13}}

\thinlines
\put(24,5){\line(1,1){4}}
\put(24,5){\line(1,-1){4}}
\put(78,5){\line(1,1){4}}
\put(78,5){\line(1,-1){4}}

\put(27.5,7){$a$}
\put(27.5,3){$b$}
\put(8,6.6){$1$}
\put(6.7,3){$\ell m$}
\put(12,6){$\ell$}
\put(22,6){$k$}
\put(40,6.5){\vector(1,0){15}}
\put(55,4.5){\vector(-1,0){15}}
\put(45,8.5){$\mathcal{A}^{-1}$}
\put(45,2){$\mathcal{A}$}

\put(82,7){$a$}
\put(81,3){$b$}
\put(75,6.7){$k$}
\put(72,3){$k\ell m$}
\end{pict}

Note that, since edge groups are cyclic, for any two non-trivial elliptic elements $a,b$ in a GBS group $\langle a \rangle \cap \langle b \rangle$ has finite index in both $\langle a \rangle$ and $\langle b \rangle$.

Another important tool in studying (G)BS groups is the so-called \emph{modular homomorphism}, see \cite{L3} and \cite{F2} for more details. The modular homomorphism $\Delta_G$ from a GBS group $G$ to the multiplicative group of rational numbers $\mathbb{Q}^*$ is defined as follows: given $g \in G$, take any elliptic element $a$ and find non-zero $p,q$ such that $ga^qg^{-1}=a^p$, and define $\Delta_G(g)=p/q$. In this notation, $\Delta_G(g)$ is called the modulus of $g$. 
One can show that the modular homomorphism is well-defined, i.e. does not depend on the particular choice of $a$ and $p,q$. Moreover, it is trivial on the elliptic elements by definition, so it factors through the quotient of $G$ by the normal subgroup generated by all the elliptic elements, which can be thought of as the (topological) fundamental group of the underlying graph $\Gamma$.

Denote by $M(G)$ the image of the modular homomorphism $\Delta_G$.

\begin{remark}\label{rem}
It is not difficult to see that if $g \in G$ corresponds to a closed path $e_1e_2 \ldots e_l$ in $\Gamma$, then $\Delta_G(g)=\frac{A(e_1)}{\Omega(e_1)}\frac{A(e_2)}{\Omega(e_2)} \dots \frac{A(e_l)}{\Omega(e_l)}$. It follows that $M(BS(m,n))= \langle \frac{n}{m} \rangle_{\mathbb{Q}^*} = \{ (\frac{n}{m})^k, \: k \in \mathbb{Z} \}$.
\end{remark}

\begin{remark} \label{r4}
If $H$ is a finite index subgroup of a GBS group $G$, then the modular homomorphism for $H$ is just the restriction of that for $G$. It follows that  $M(H)$ is a finite index subgroup in $M(G)$. Thus, for any finite index subgroup $H$ of $BS(m,n)$ we have $M(H)= \langle (\frac{n}{m})^k \rangle_{\mathbb{Q}^*}$ for some $k \in \mathbb{Z}_+$. Therefore, if $BS(m,n)$ and $BS(p,q)$ are commensurable, then there exist $k, l \in \mathbb{Z} \backslash \{ 0 \}$ such that $(\frac{n}{m})^k=(\frac{q}{p})^l$. In particular, if $m_1 \mid n_1$, but $m_2 \nmid n_2$, then $BS(m_1,n_1)$ and $BS(m_2,n_2)$ are not commensurable.
\end{remark}

A GBS group $G$ is said to have \emph{no non-trivial integral moduli} (as defined in \cite{F2}) if the image of the modular homomorphism $M(G)$ contains no integers of absolute value bigger than 1. The class of GBS groups with no non-trivial integral moduli is better understood than GBS groups in general. In particular, the isomorphism problem for such GBS groups was solved by Forester in \cite{F2}. One of the reasons why such GBS groups are better behaved is that the deformation space for such groups is non-ascending, i.e. any graph in such space has no strictly ascending loops (a loop is called \emph{strictly ascending} if one of its labels is $\pm 1$ and the other is not), see \cite{CF1} and \cite{GL}. It is proved in \cite{CF1} that in a non-ascending deformation space any two reduced graphs are related by slide moves, while in general deformation spaces (even for GBS groups) induction moves and $\mathcal{A}^{\pm 1}$-moves are necessary, see \cite{CF1}, \cite{F2}, \cite{F3}. In particular, in a non-ascending  deformation space any two reduced graphs have the same number of vertices and edges, which might be no longer true in the ascending case. We refer the reader to  \cite{L4} for a simple example (in the case of $BS(2,6)$). Note that if $G$ is a GBS group with no non-trivial integral moduli, and $H$ is a finite index subgroup of $G$, then $H$ also has no non-trivial integral moduli, by Remark \ref{r4}.

It is well-known that Baumslag-Solitar group $BS(m,n)$ is solvable if and only if $m$ or $n$ is equal to $\pm 1$, and otherwise it contains a non-abelian free subgroup, so is not quasi-isometric to a solvable one. 
Since $BS(-m,-n) \cong BS(m,n) \cong BS(n,m)$, we can always assume that $n \geq |m| \geq 1$.

\section{Finite index subgroups of (G)BS groups}

Finite index subgroups of Baumslag-Solitar groups $BS(p,q)$ with $\gcd(p,q)=1$, and, more generally, subgroups of GBS groups without proper plateau, have a very nice description: they correspond to covers of the graph $\Gamma$ defining the group (with lifted labels), see \cite[Corollary 6.6]{L2}. 

The notion of {\it plateau} was introduced in \cite[Definition 3.1]{L2} and we recall it here for completeness. For a prime number $p'$ a non-empty connected subgraph $P$ of a GBS graph $\Gamma$ is a $p'$-plateau if the following condition holds: for every edge $e$ which starts in a vertex of $P$ the label $A(e)$ is divisible by $p'$ if and only if $e$ is not contained in $P$. The graph $P$ is called a plateau if it is a $p'$-plateau for some prime $p'$. It is proper if it doesn't coincide with the whole $\Gamma$.

In particular, one can completely describe finite index subgroups of $BS(p,q)$ when $\gcd(p,q)=1$, as follows. 
  
\begin{prop}[see \cite{D1, L2,W}] Let $H$ be a finite index subgroup of the Baumslag-Solitar group $BS(p,q)$, where $\gcd(p,q)=1$. Then $H$ is a GBS group given by a cycle $e_1e_2 \ldots e_l$, $l \geq 1$, and such that $A(e_i)=p$, $\Omega(e_i)=q$, for each $i=1, \dots, l$. Moreover, this cycle is the induced (from the action of $H$ on the Bass-Serre tree of $BS(p,q)$) graph of groups decomposition of $H$.
\end{prop}

For general (G)BS groups this description of finite index subgroups does not hold. However, in \cite[Proposition 6.8]{L2}, Levitt shows that every GBS group contains a finite index subgroup without proper plateau. In order to study finite index subgroups of Baumslag-Solitar groups $BS(m,n)$ when $n>m>1$ and $\gcd(m,n)=d>1$, we first describe a special finite index subgroup of this group without proper plateau.

Let $n>m>1$ and $\gcd(m,n)=d$, $m=dp$, $n=dq$, so $\gcd(p,q)=1$. Define a homomorphism 
$$
\varphi=\varphi_{m,n}: BS(m,n) \rightarrow \langle z \rangle _d, \: \varphi (t) = 1, \: \varphi (a) = z. 
$$
This is well-defined since $n-m$ is a multiple of $d$. Let $H_{m,n}=Ker (\varphi_{m,n})$, which has index $d$ in $BS(m,n)$.

Define the following group:
\begin{equation} \label{eq:gdpq}
G^d_{p,q}=\langle a, t_1, \ldots, t_k \mid t_i^{-1}a^pt_i=a^q, \: i=1, \ldots, k \rangle,
\end{equation}
i.e., $G^d_{p,q}$ is the GBS group with underlying graph having one vertex and $d$ directed loops $e_1, \ldots, e_d$ such that $A(e_i)=p$, $\Omega(e_i)=q$.

The following lemma describes the structure of the subgroup $H_{m,n}$: it is isomorphic to the group $G_{p,q}^d$.

\begin{lemma}\label{l1}
    Let $H_{m,n}$ be a finite index subgroup of $BS(m,n)$ defined as above.
    In the above notation, the following holds.
    \begin{enumerate}
        \item $H_{m,n} = \langle \langle t, a^d \rangle \rangle$.
        \item $H_{m,n} = \langle a^d, t, t^a, \ldots, t^{a^{d-1}} \rangle $.
        \item $H_{m,n}$ is isomorphic to the GBS group $G^d_{p,q}$.  
        \item The subgroups $B_i= \langle a^d, t^{a^i} \rangle$ are all isomorphic to $BS(p,q)$, for $i=0, \ldots, d-1$.
    \end{enumerate}
\end{lemma}
    \begin{proof}
       Denote $H=H_{m,n}$ and $G=BS(m,n)$ for short. The first claim follows directly from the definition of $H$.
       
       Denote $H_0 = \langle a^d, t, t^a, \ldots, t^{a^{d-1}} \rangle \subseteq H$. We want to show that $H_0=H$. Since $t, a^d \in H_0$, it suffices to show that $H_0$ is normal, i.e. that for any $g \in G$ we have $t^g \in H_0$ and $(a^d)^g \in H_0$. 
       
       Note that $t^{a^k} \in H_0$ for all $k$. Indeed, let $k=ds+r$, where $0 \leq r \leq d-1$, then $$t^{a^k}=a^kta^{-k}=(a^d)^s(a^rta^{-r})(a^d)^{-s}=(a^d)^st^{a^r}(a^d)^{-s} \in H_0.$$

       Suppose that $g=a^{k_1}t^{\epsilon_1}a^{k_2}t^{\epsilon_2} \ldots a^{k_N}t^{\epsilon_n}a^{k_{N+1}}$, with $\epsilon_i \in \{ \pm 1 \}$ and $k_i \in \mathbb{Z}$. Then
       $$
       g= (t^{\epsilon_1})^{a^{k_1}} a^{k_1+k_2}t^{\epsilon_2} \ldots a^{k_N}t^{\epsilon_N}=(t^{\epsilon_1})^{a^{k_1}}(t^{\epsilon_2})^{a^{k_1+k_2}}\ldots(t^{\epsilon_N})^{a^{k_1+k_2+\ldots+k_N}}a^M=h_0 a^C, 
       $$
       where $M=k_1+k_2+\ldots+k_{N+1}$, $M=Sd+C$, $h_0 \in H_0$ and $0 \leq C \leq d-1$.
       Then $t^g=h_0a^Cta^{-C}h_0^{-1}=h_0t^{a^C}h_0^{-1} \in H_0$ and
       $(a^d)^g=h_0a^Ca^da^{-C}h_0^{-1}=h_0a^dh_0^{-1} \in H_0$. This proves the second claim.
       
       Now consider the action of $G=BS(m,n)$ on the Bass-Serre tree $T=T_G$ associated to the standard splitting of $G$. By Bass-Serre theory, the vertex set of $T$ is the set of right cosets $g \langle a \rangle$, $g \in G$, and the edges incident to the vertex $\langle a \rangle$ are the following: outgoing edges $e_0, e_1, \ldots, e_{m-1}$,  with $e_i$ going to the vertex $a^i t\langle a \rangle$, $i=0, \ldots, m-1$, and ingoing edges $f_0, f_1, \ldots, f_{n-1}$, with $f_j$ going from the vertex $a^j t^{-1} \langle a \rangle$, $j=0, \ldots, n-1$. The action is by left multiplication and all the other stars of vertices in $T_G$ are obtained by transferring the star of the vertex $\langle a \rangle$ by an element of $G$.
       
       Recall that $m=pd$ and $n=qd$. We claim that the set of edges $e_0, e_1, \ldots, e_{d-1}$ together with the vertex $w_0=\langle a \rangle$ is a fundamental domain for the action of $H$ on $T$.
       
       Indeed, first notice that $H$ acts transitively on the set of vertices of $T$: since $H, Ha, Ha^2, \ldots, Ha^{d-1}$ forms the set of cosets of $H$ in $G$ by definition of $H$, for any $g \in G$ we have $g=ha^k$ for some $0 \leq k \leq d-1$, $h \in H$, so $g \langle a \rangle = h \langle a \rangle$ and $h^{-1}$ takes $g \langle a \rangle$ to $\langle a \rangle$. It follows that any edge $e$ in $T$ can be taken by $H$ to one of the edges from the set $e_0, e_1, \ldots, e_{m-1}$, by taking the initial vertex of $e$ to $w_0$,  since the edge orientations are preserved by the action.
       
       Now we show that the edges $e_i$ and $e_j$ are in the same $H$-orbit if and only if $d \mid (i-j)$. Indeed, if $i-j=sd$ for some $s$ then the element $(a^d)^s \in H$ takes $e_j$ to $e_i$, since it fixes $w_0$ and takes the vertex $a^jt\langle a \rangle$ to the vertex $a^it \langle a \rangle$. On the other hand, if $he_i=e_j$ for some $h \in H$ then $h$ fixes $w_0$, so $h=a^{dl}$ for some $l$, and $ha^it\langle a \rangle =a^jt \langle a  \rangle$, so $a^{dl}=h=a^jta^Nt^{-1}a^{-i}$, for some $N$, therefore, $a^N=t^{-1}a^{i-j+dl}t$. This can happen in $BS(m,n)$ only if $d \mid i-j+dl$, so $d | i-j$, as desired.
       
       This shows that indeed the set of edges $e_0, e_1, \ldots, e_{d-1}$ together with the vertex $w_0=\langle a \rangle$ is a fundamental domain for the action of $H$ on $T$, and so the quotient graph $H \backslash T$ has one vertex and $d$ loops $E_1, \ldots, E_d$.
       
       Note that the stabilizer of $w_0$ in $H$ is equal to $\langle a^d \rangle$. If $u_i$ is the end vertex of $e_i$, $i=0, \ldots, d$, then the stabilizer of $u_i$ in $H$ is equal to $\langle ta^dt^{-1} \rangle$, since $H$ is a normal subgroup. The stabilizer of  the edge $e_i$ in $H$ (and in $G$) is equal to $\langle a^m \rangle = \langle ta^nt^{-1} \rangle$, which has index $p$ in $\langle a^d \rangle$ and index $q$ in $\langle ta^dt^{-1} \rangle$. It follows that $A(E_i)=p$ and $\Omega(E_i)=q$ for all $i=1, \ldots, d$, as desired. This proves the third claim.
       
       Finally, note that the element $t^{a^i}$ takes $w_0$ to $u_i$, so it can be taken as the Bass-Serre element corresponding to the edge $E_i$ in the induced splitting of $H$. It follows immediately that  $B_i= \langle a^d, t^{a^i} \rangle \cong BS(p,q)$, for $i=0, \ldots, d-1$, and the lemma is proven.
    \end{proof}

Thus, the question of commensurability of Baumslag-Solitar groups reduces to the question of commensurability of the groups $G^d_{p,q}$ with $gcd(p,q)=1$. 
Since the group $G^d_{p,q}$ does not have proper plateaus when $gcd(p,q)=1$, we can give a description of its finite index subgroups, which follows from \cite{L2} (one can also deduce it from \cite{D1}).

The following lemma states that finite index subgroups of $G^d_{p,q}$ can be all obtained through graph coverings of the bouquet of circles, see Figure \ref{fig:1} for an example of a finite index subgroup $K$ of the group $G_{1,p}^2$.

\begin{lemma}\label{l3}
    Let $K$ be a finite index subgroup of $G^d_{p,q}$, for some $d,p,q \geq 1$, with $gcd(p,q)=1$. Then $K$ is isomorphic to a GBS group defined by a GBS graph $\Gamma_K$ such that 
    \begin{enumerate}
            \item There is a map $\pi$ from $\Gamma_K$ to the bouquet of circles defining $G_{p,q}^d$, which is a covering map of directed graphs. In particular, each vertex of $\Gamma_K$ has degree $2d$, with $d$ incoming and $d$ outgoing positive edges. 
            \item For every positive edge $e$ in $\Gamma_K$ we have $A(e)=p$, $\Omega(e)=q$. In particular, if $q > p > 1$ then $\Gamma_K$ is reduced.
            \item The quotient of $K$ over the subgroup of all elliptic elements is a free group, defined by  $\Gamma_K$ as a covering of a bouquet of $d$ circles, i.e. it is isomorphic to the topological fundamental group of $\Gamma_K$.
        \end{enumerate}
\end{lemma}
\begin{proof}
   Note that since $p$ and $q$ are coprime, the GBS graph defining $G^d_{p,q}$ (which is a bouquet of circles) contains no proper plateau. By \cite[Corollary 6.6]{L2}, it follows that every finite index subgroup $K$ of $G^d_{p,q}$ can be represented by a labelled graph $\Gamma_K$ which is a (topological) covering of the bouquet of circles defining $G_{p,q}^d$. This proves the first claim.
   
   It follows from \cite{L2} that the corresponding covering map $\pi$ is an admissible map in the sense of \cite[Definition 6.1]{L2}, and since $\pi$ is a covering, it follows from \cite[Lemma 6.4]{L2} that $\pi$ preserves the labels. This proves the second claim.
   
    The third claim is well-known and easy to see, see \cite{L3}.
\end{proof}

   In fact, it is not hard to see that the graph $\Gamma_K$ is given by the induced splitting of $K$ under its action on the Bass-Serre tree of $G_{p,q}^d$ with respect to the natural splitting, see \cite{L2}.

\section{Structure of finite index subgroups of $G_{1,q}^d$}

We now turn to the case when $p=1, q>1, d>1$. In this section, we give a canonical way to describe finite index subgroups of $G_{1,q}^d$ and show in Section \ref{sec:iso} that this representation is in fact a normal form that allows us to solve the isomorphism problem for finite index subgroups of $G_{1,q}^d$ in an explicit way.

The subgroup $K$ still has the GBS structure described by Lemma \ref{l3}, but in the case we are considering the GBS graph $\Gamma_K$ might be not reduced, so we can apply collapsing moves to it, as the following lemma describes.

\begin{lemma}\label{collapse}
 Let $K$ be a finite index subgroup of $G^d_{1,n}$, for some $d,n \geq 2$, and $\Gamma_K$ be the corresponding covering GBS graph given by Lemma \ref{l3}. Let $k$ be the rank of the topological fundamental group of $\Gamma_K$. Then 
 \begin{enumerate}
 \item \label{collapseit1} There exists a spanning tree $S$ in $\Gamma_K$ with all the positive edges in the tree oriented towards a given vertex.
 \item \label{collapseit2} $K$ is isomorphic to a GBS group defined by a bouquet $B_K$ of $k$ circles $e_1, \ldots, e_k$, which is obtained by collapsing all the edges in $S$. More precisely, for each $i=1, \ldots,k$ the edge $e_i$ is the image under collapsing of $z_i$, where $z_1, \ldots, z_k$ are the reduced paths in $\Gamma_K$ generating the topological fundamental group of $\Gamma_K$ such that each path $z_i$ contains exactly one edge outside of $S$, which is positive.
  \item \label{collapseit3}  If $p(z_i)$ is the number of positive edges  in $z_i$, and $q(z_i)$ is the number of negative edges in $z_i$, then $A(e_i)=n^{q(z_i)}$ and $\Omega(e_i)=n^{p(z_i)}$, for each $i=1, \ldots, k$.
  \item \label{collapseit4} There exist $i_1 \neq i_2 \in \{ 1, \ldots, k \}$ such that $A(e_{i_1})=A(e_{i_2})=1$, and $\Omega(e_{i_1}) > 1$.
  \item \label{collapseit5} If $K \neq G^d_{1,n}$ then $k \geq 3$ and there exists $j \in \{1, \ldots, k \}$ such that $A(e_j)=n$. 
  \end{enumerate}
\end{lemma}
\begin{proof}
        Recall that by Lemma \ref{l3} the subgroup $K$ can be represented by a GBS graph $\Gamma_K$ covering a bouquet of circles which defines $G_{p,q}^d$. We denote by $D_i$ the set of vertices of $\Gamma_K$ projecting into a given edge $e_i$ of $G_{p,q}^d$, for $i=1, \ldots, d$. Thus, each connected component of  $D_i$ is a simple directed cycle, and each vertex of $\Gamma_K$ belongs to exactly one such cycle.       Note that $\Gamma_K$ is connected as a directed graph, i.e. for any vertices $u$, $v$ of $\Gamma_K$ there is a path from $u$ to $v$ which traverses all edges in the positive direction.
        
        We now construct a spanning tree as in \eqref{collapseit1}. Start with any vertex $v$ and let $S_1$ be just the vertex $v$. Let $N$ be the number of vertices of $\Gamma_K$. We claim that for each $i=1, \ldots, N$ there exists a subtree $S_i$ of $\Gamma_K$ with all edges oriented towards $v$ which has $i$ vertices. Indeed, by induction let it exist for $i \leq i_0 < N$. Let $w$ be any vertex of $\Gamma_K$ outside of $S_{i_0}$. Let $p$ be a path from $w$ to $v$ consisting of only positive edges, which exists since $\Gamma_K$ is connected as a directed graph. Let $u$ be the last vertex on this path which is not in $S_{i_0}$, and $e$ be the edge coming after $u$ in $p$. Then we can add $u$ together with $e$ to $S_{i_0}$ and obtain $S_{i_0+1}$. In the end we get $S_N$ which is the required spanning tree, so \eqref{collapseit1} holds.
        
        We now prove \eqref{collapseit2}. Recall that for every edge $e$ of $\Gamma_K$ we have $A(e)=1$ and $\Omega(e)=n$, by Lemma \ref{l3}. We apply collapsing moves to the edges of $S$ in any order. When we collapse an edge $e$ with $A(e)=1$ and $\Omega(e)=C$ for some $C \geq 1$, beginning in $u_1$ and ending in $u_2$, the vertices $u_1$ and $u_2$ get identified, all the labels of edges incident to $u_1$ at $u_1$ get multiplied by $C$, and all the other labels remain unchanged. Note that by definition of $S$ at every vertex of $S$ except the base vertex $v$ there is just one outgoing positive edge which is in $S$, and there is no outgoing positive edge in $S$ from $v$. Therefore, after collapsing any edge in $S$, all the remaining positive edges in $S$ will still have an $A$-label equal to 1. This implies that we can continue collapsing until all the edges from $S$ are collapsed, and what we get is a bouquet of $k$ circles $e_1, \ldots, e_k$. It is immediate that $e_i$ is the image under collapsing of $z_i$, as in the statement, for every $i=1, \ldots, k$. This proves \eqref{collapseit2}.
        
        We now show \eqref{collapseit3}. Consider the path $z_i$ for some fixed $i=1, \ldots, k$. Note that $z_i=y_i^{-1}E_ix_i$, where $E_i$ is a positive edge outside of $S$ and $x_i, y_i$ are paths in $S$ with all edges positive (each of them might be empty; they might also have edges in common). Then $p(z_i)=| x_i | +1$ and $q(z_i)=|y_i|$, where $|p|$ denotes the edge-length of a path $p$. It is easy to see that collapsing all the edges in $y_i$ makes $A(E_i)$ multiply by $n^{|y_i|}$. 
        Similarly, collapsing all the edges in $x_i$ makes $\Omega(e_i)$ multiply by $n^{|x_i|}$. Collapsing all the other edges in $S$ does not affect the labels of $E_i$; also, collapsing an edge which is in $x_i$ but not in $y_i$ does not affect $A(E_i)$ and collapsing an edge which is in $y_i$ but not in $x_i$ does not affect $\Omega(E_i)$. 
        Since $A(E_i)=1$, $\Omega(E_i)=n$, we get $A(e_i)=n^{|y_i|}=n^{q(z_i)}$ and $\Omega(e_i)=n^{|x_i|+1}=n^{p(z_i)}$, as required. This proves \eqref{collapseit3}.
        
        Since $d \geq 2$, there are at least two positive edges beginning in the base vertex $v$, and all positive edges beginning in $v$ are not in $S$. It follows that there are at least two paths among $z_1, \ldots, z_k$ which start from an edge outside of $S$. Suppose these are $z_{i_1}$ and $z_{i_2}$. These paths consist of only positive edges, so $q(z_{i_1})=q(z_{i_2})=0$, and it follows from \eqref{collapseit3} that $A(e_{i_1})=A(e_{i_2})=1$. It is immediate also that $\Omega(e_{i_1}), \Omega(e_{i_2})>1$. This proves \eqref{collapseit4}.
        
        Now suppose that $K \neq G^d_{1,n}$. It follows that $\Gamma_K$ has at least two vertices, since if it only had one vertex, it would have $d$ oriented loops at this vertex with all the $A$-labels equal to $1$ and all the $\Omega$-labels equal to $n$, and so we would have $K=G^d_{1,n}$. It follows that $k \geq 3$, since $d \geq 2$.
        Also, there is at least one positive edge in $S$. Let $e$ be a positive edge of $S$ ending in the base vertex $v$, and suppose $e$ starts in a vertex $w$. Let $e'$ be any positive edge starting in $w$ distinct from $e$, such an edge exists since $k \geq 2$. Then $e' \notin S$ by the definition of $S$. Thus there exists a path among $z_1, \ldots, z_k$, which is of the form $z_j=e^{-1}e'p$, where $p$ is some path in $S$ with all edges positive. Then $q(z_j)=1$, and so by \eqref{collapseit3} $A(e_j)=n$, as required. This proves \eqref{collapseit5} and the lemma, see Figure \ref{fig:1}.
\end{proof}

\begin{center}
\begin{figure}[ht]
\begin{tikzpicture}
    \node (x) at (0,0) [point];
    \node (y) at (2,0) [point];
    \node (z) at (0,1.5) [point];
    \node (t) at (2,1.5) [point];
    \node (vv) at (2.2,-0.2) {\footnotesize $v$};

    \path[->] (y) edge (x);
    \path[->,line width=1.25] (x) edge (z);
    \path[->,line width=1.25] (z) edge (t);
    \path[->,line width=1.25] (t) edge (y);
   
    \path[->,dashed] (z) edge[bend left=50] (x);
    \path[->,dashed] (x) edge[bend left=50] (z);
    \path[->,dashed] (t) edge[bend left=50] (y);
    \path[->,dashed] (y) edge[bend left=50] (t);
    
   \node (x1) at (-0.3,0.1) {\footnotesize $1$};
   \node (x2) at (-0.1,0.3) {\footnotesize $1$};
   \node (x3) at (0.45,0.3) {\footnotesize $p$};
   \node (x4) at (0.3,-0.2) {\footnotesize $p$};

   \node (y1) at (1.8,-0.2) {\footnotesize $1$};
   \node (y2) at (1.6,0.3) {\footnotesize $1$};
   \node (y3) at (2.15,0.5) {\footnotesize $p$};
   \node (y4) at (2.4,0.2) {\footnotesize $p$};
   
   \node (z1) at (-0.4,1.3) {\footnotesize $p$};
   \node (z2) at (0.35,1.3) {\footnotesize $1$};
   \node (z3) at (0.15,1) {\footnotesize $p$};
   \node (z4) at (0.2,1.7) {\footnotesize $1$};

   \node (t1) at (1.5,1.7) {\footnotesize $p$};
   \node (t2) at (2.3,1.4) {\footnotesize $1$};
   \node (t3) at (2.1,1.2) {\footnotesize $1$};
   \node (t4) at (1.8,1.1) {\footnotesize $p$};

  \node (g1) at (1, -0.5){};
  \node (g2) at (1, -1.8){};
  \path[line width=1.5] (g1) edge (g2);

  \node (u) at (1,-3) [point];

 \path (u) edge[scale=3,in=-150,out=150,loop] (u);
 \path[dashed] (u) edge[scale=3,in=30,out=-30,loop] (u);

   \node (u1) at (0.9,-2.75) {\footnotesize $1$};
   \node (u2) at (0.8,-3.3) {\footnotesize $p$};
   \node (u3) at (1.25,-2.75) {\footnotesize $p$};
   \node (u4) at (1.25,-3.3) {\footnotesize $1$};

 \node (l33) at (1,2.2) {$\Gamma_K$ in Lemma \ref{l3}};
 \node (l322) at (1,-3.8) {\footnotesize Defining graph of $G_{1,p}^2$};

  \node (g3t) at (3.7, 0.9){\footnotesize reducing};
  \node (g3) at (2.7, 0.75){};
  \node (g4) at (5, 0.75){};
  \path[line width=1.5] (g3) edge (g4);
 
 \node (v) at (6.5,0.75) [point];
 \path (v) edge[scale=3,out=0,in=60, loop] (v);
 \path[dashed] (v) edge[scale=3,out=70,in=130, loop] (v);
 \path[dashed] (v) edge[scale=3,out=140,in=200, loop] (v);
 \path[dashed] (v) edge[scale=3,out=210,in=270, loop] (v);
 \path[dashed] (v) edge[scale=3,out=280,in=340, loop] (v);

\node (v1) at (6.7, 1.5){\footnotesize $p$};
\node (v11) at (6,1.3){\footnotesize $p$};

\node (v3) at (7.2, 0.7){\footnotesize $p^4$};
\node (v31) at (7.2, 1.5){\footnotesize $1$};

\node (v2) at (6.7, -0.1){\footnotesize $p^3$};
\node (v21) at (7.5, 0.25){\footnotesize $p^3$};

\node (v4) at (5.75, 0.1){\footnotesize $p^2$};
\node (v41) at (6.35, -0.2){\footnotesize $p^4$};

\node (v5) at (5.8, 1.25){\footnotesize $1$};
\node (v51) at (5.45, 0.5){\footnotesize $p^2$};

\node (l41) at (6.5,2.2) {Lemma \ref{collapse}};

\node (g5) at (6.5, -0.5){};
\node (g6) at (6.5, -1.8){};
 
 \path[line width=1.5] (g5) edge (g6);
\node (w) at (7.05,-0.85) {\footnotesize normal};
\node (w) at (7.05,-1.1) {\footnotesize form};
\node (w) at (0.3,-1.1) {\footnotesize cover};

 \node (w) at (6.5,-3) [point];

 \path (w) edge[scale=3,out=0,in=60, loop] (w);
 \path[dashed] (w) edge[scale=3,out=70,in=130, loop] (w);
 \path[dashed] (w) edge[scale=3,out=140,in=200, loop] (w);
 \path[dashed] (w) edge[scale=3,out=210,in=270, loop] (w);
 \path[dashed] (w) edge[scale=3,out=280,in=340, loop] (w);

\node (w1) at (6.7, -2.25){\footnotesize $p$};
\node (w11) at (6,-2.35){\footnotesize $p$};

\node (w3) at (7.2, -3.05){\footnotesize $1$};
\node (w31) at (7, -2.25){\footnotesize $1$};

\node (w2) at (6.9, -3.95){\footnotesize $p$};
\node (w21) at (7.45, -3.5){\footnotesize $p$};

\node (w4) at (5.75, -3.65){\footnotesize $1$};
\node (w41) at (6.35, -3.95){\footnotesize $1$};

\node (w5) at (5.8, -2.5){\footnotesize $1$};
\node (w51) at (5.45, -3.25){\footnotesize $p^2$};

\node (l41) at (6.5,-4.3) {Lemma \ref{slide}};

\end{tikzpicture}
\caption{Example illustrating Lemmas \ref{l1}, \ref{l3}, \ref{collapse} and \ref{slide}. The bold edges in the top left graph represent the spanning tree $S$, as in Lemma \ref{collapse}.} \label{fig:1}
\end{figure}
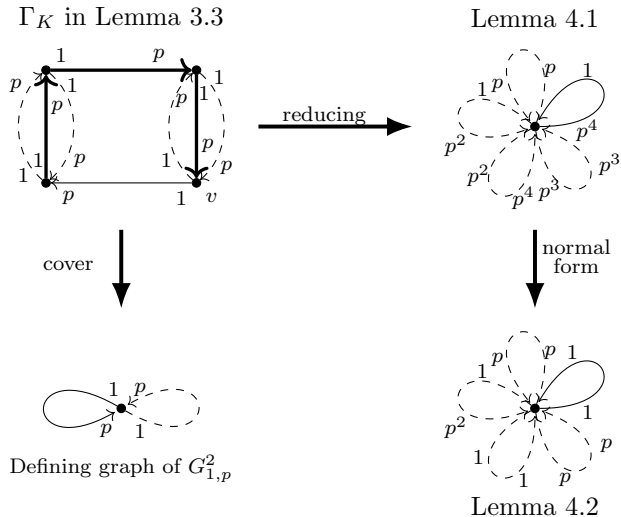
\end{center}

After collapsing to a bouquet of circles, as described by Lemma \ref{collapse}, we can adjust the petal labels by applying slide moves, as described in the following lemma, see Figure \ref{fig:1} for an example.

In the notation of Lemma \ref{collapse}, let
  \begin{equation}\label{m}
  m=\gcd (|p(z_1)-q(z_1)|, |p(z_2)-q(z_2)|, \ldots, |p(z_k)-q(z_k)|).
  \end{equation}
  It follows from Lemma \ref{collapse} that there exists $i$ such that $q(z_i)=0$ and $p(z_i)>0$, so $m \geq 1$.
\begin{lemma}\label{slide}
 In the above notation, the following holds.
 The subgroup $K$ is isomorphic to a GBS group defined by a bouquet $B'_K$ of circles $f_1, \ldots, f_k$, such that
\begin{enumerate}
     \item \label{slideit1}
 $A(f_1)=1$, $\Omega(f_1)=n^m$.
  \item \label{slideit2}
  For each $i=2, \ldots, k$ we have $A(f_i)=\Omega(f_i)=n^{p_i}$, where $0 \leq p_i \leq m-1$. 
  \item \label{slideit3}
 There exists $i_1 \in \{2, \ldots, k \}$ such that $A(f_{i_1})=\Omega(f_{i_1})=1$. 
  \item  \label{slideit4} 
If also $K \neq G^d_{1,n}$ and $m \geq 2$, then $k \geq 3$ and there exists $i_2 \in \{2, \ldots, k \}$ such that $A(f_{i_2})=\Omega(f_{i_2})=n$.
 \end{enumerate}
\end{lemma}
\begin{proof}
We know that $K$ is isomorphic to a GBS group defined by a bouquet $B_K$ of $k$ circles $e_1, \ldots, e_k$, described by Lemma \ref{collapse}. By Lemma \ref{collapse}\eqref{collapseit4}, in which without loss of generality we can assume that $i_1=1$, we have $q(z_1)=0$, and $p(z_1)>0$. Let $m_1=p(z_1)$ and $m_i=\\gcd(m_{i-1}, |p(z_i)-q(z_i)|)$, for $i=2, \ldots, k$. 
Note that $m_k=m$, and that each $m_i$, $1 \leq i \leq k$, is a multiple of $m$.

The idea is to apply multiple slide moves to $B_K$ in order to reach the desired bouquet of circles $B_K'$. We can successively apply these slides for pairs $(e_1,e_2), (e_1,e_3), \ldots, (e_1,e_k)$ to make $e_1$ into a loop with $A$-label equal to $1$ and $\Omega$-label equal to $n^m$, and then slide all the other loops over it to make their $A$- and $\Omega$- labels as described in \eqref{slideit2}.

We first describe abstractly the procedure we are going to apply later to each pair $(e_1,e_i)$, $i=2, \ldots, k$. This is a modification of Euclid's algorithm. Suppose that we have a GBS group which includes a vertex $w$ with 2 loops, $E$ and $F$, with $A(E)=1$, $\Omega(E)=n^a$, $A(F)=n^b$ and $\Omega(F)=n^c$, where $a \geq 1$, $b,c \geq 0$. Denote $a_0=a, b_0=b, c_0=c, E_0=E, F_0=F$. Apply the following for $i=0$.

{\it Step 1}.
Let $b_i=a_is_i+b_{i+1}$, $c_i=a_it_i+c_{i+1}$, where $0 \leq b_{i+1}, c_{i+1} \leq a_i-1$. Slide the edge $F_i$ over the edge $E_i$ $s_i$ times on the one side (of $F_i$) and $t_i$ times on the other side (of $F_i$), so that the obtained (from $F_i$) edge $F_{i+1}$ has $A(F_{i+1})=n^{b_{i+1}}$ and $\Omega(F_{i+1})=n^{c_{i+1}}$, while the edge $E_i$ is not affected. Now, if $b_{i+1}=c_{i+1}$ then let $E_{i+1}=E_i$,  $a_{i+1}=a_i$ and terminate the procedure. 

{\it Step 2}. If $b_{i+1}>c_{i+1}$, let $a_i-c_{i+1}=\lambda_i (b_{i+1}-c_{i+1})+\mu_i$, where $1 \leq \mu_i \leq (b_{i+1}-c_{i+1})$, let $a_{i+1}=a_i-\lambda_i (b_{i+1}-c_{i+1})=\mu_i+c_{i+1} \geq 1$, and slide the edge $E_i$ over the edge $F_{i+1}$ $\lambda_i$ times, so that the obtained (from $E_i$) edge $E_{i+1}$ has $A(E_{i+1})=1$ and $\Omega(E_{i+1})=n^{a_{i+1}}$, while the edge $F_{i+1}$ is not affected. Note that in this case $1 \leq a_{i+1} \leq b_{i+1}$. If $c_{i+1}>b_{i+1}$ apply the same but with the roles of $b$'s and $c$'s interchanged. 

Now keep applying the above procedure (both steps) with $i$ increased by $1$ at each step, until it terminates. Note that for each $i$ we have $a_i > b_{i+1}$, so (unless the procedure terminates at this step) we have $a_i-c_{i+1} > b_{i+1}-c_{i+1}$ and so $\lambda_i \geq 1$, therefore, $a_{i+1} \leq a_i-1$, and so the procedure has to terminate.

    Let $I$ be the index for which the procedure terminates and $N=I+1$. We let $(E_N, F_N)$, with their labels, be the output of our procedure. It is immediate that $A(E_N)=1$ and $A(F_N)=\Omega(F_N)$. Note that for all $0 \leq i \leq N$ we have $A(E_i)=1$, $\Omega(E_i)= n^{a_i}$, $A(F_i)=n^{b_i}$ and $\Omega(F_i)=n^{c_i}$.

    Let $d=\gcd(a,|b-c|)$. We claim that $a_N=d$,
    $0 \leq b_N=c_N \leq d-1$ and $b_N = b \mod d$. 
    
    We first show that $a_N \mid a$ and $a_N \mid |b-c|$. We claim that $a_N \mid a_i$, $a_N \mid |b_i-c_i|$, $i=0, \ldots, N$. It is obvious for $i=N$, since $b_N=c_N$. Suppose by (reverse) induction it is true for $i=j+1$, where $0 \leq j \leq N-1$, then 
    $$
    \begin{array}{ll}
    a_j=a_{j+1}+\lambda_j \cdot|b_{j+1}-c_{j+1}|,& \hbox{ so } a_N \mid a_j,\hbox{ and}\\
    b_{j}-c_{j}=a_j(s_{j}-t_{j})+(b_{j+1}-c_{j+1}),& \hbox{ so } a_N \mid |b_{j}-c_{j}|.
  \end{array}
  $$
  For $i=0$ we get  $a_N \mid a$ and $a_N \mid |b-c|$.

    Suppose $D$ is such that $D \mid a$ and $D \mid |b-c|$. We show that in this case $D \mid a_N$. Indeed, $D \mid a_0$ and $D \mid |b_0-c_0|$, and we can see by induction that $D \mid a_i, |b_i - c_i|$, for $0 \leq i \leq N$: suppose it is true for $i=j$, $0 \leq j \leq N-1$, then we have 
    $$
    \begin{array}{ll}
    b_{j+1}-c_{j+1}=(b_j-c_j)-a_j(s_j-t_j),& \hbox{ so } D \mid b_{j+1}-c_{j+1},\hbox{ and}\\
    a_{j+1}=a_j-\lambda_j | b_{j+1}-c_{j+1}| ,& \hbox{ so } D \mid a_{j+1}.
    \end{array}
    $$
    Therefore, $D \mid a_N$. Since $a_N \mid a$ and $a_N \mid |b-c|$, it follows that $a_N=d$.

    Note that $0 \leq b_N=c_N \leq a_{N-1}-1=a_N-1=d-1$ by construction. Moreover, since $b_i=a_is_i+b_{i+1}$, and $d \mid a_i$, $i=0, \ldots, N-1$, we have that $b_i = b_{i+1} \mod d$. Since $b_0=b$, we get that $b_n = b \mod d$, as required. The same holds with $c$'s instead of $b$'s.
     Thus, $A(E_N)=1$, $\Omega(E_N)=n^d$, and $A(F_N)=\Omega(F_N)=n^l$, where $l$ is the residue of $b$ (and of $c$) modulo $d$.
    
    Recall that we have $A(e_1)=1$, $\Omega(e_1)=n^{p(z_1)}$ with $p(z_1) \geq 1$, $A(e_2)=n^{q(z_2)}$, $\Omega(e_2)=n^{q(z_2)}$. We now apply the above procedure (i.e., corresponding slide moves) to $e_1$ and $e_2$ in $B_K$, without affecting the other edges. We get a bouquet of circles $e_1^{2}, e_2', e_3, \ldots, e_k$, such that 
    $$
    A(e_1^2)=1, \; \Omega(e_1^2)=m_2=\gcd(p(z_1), |p(z_2)-q(z_2)|),  \; A(e_2')=\Omega(e_2')=n^{l_2}, 
    $$
    where $l_2 = p(z_2)=q(z_2) \mod m_2$, $0 \leq l_2 \leq m_2-1$. Then apply the above procedure to $(e_1^2, e_3)$, getting a bouquet of circles  $(e_1^3, e_2', e_3', e_4, \ldots, e_k)$, with $A(e_1^3)=1$, $\Omega(e_1^3)=n^{m_3}$, $A(e_3')=\Omega(e_3')=n^{l_3}$, where $l_3=p(z_3)=q(z_3) \mod m_3$,  $0 \leq l_3 \leq m_3-1$.
   Continue successively applying the procedure to all the pairs $(e_1^i,e_{i+1})$, $i=1, \ldots,k-1$. 
   In the end we get a bouquet of circles $(f_1,e_2',e_3', \ldots, e_k')$ such that the following holds:  $A(f_1)=1$, $\Omega(f_1)=n^{m}$, $A(e_i')=\Omega(e_i')=n^{l_i}$, where $l_i=p(z_i)=q(z_i) \mod m_i$,  $0 \leq l_i \leq m_i-1$, for all $i=2,\ldots, k$. Here $f_1=e_1^k$.
   
   The following applies to every $i=2, \ldots, m$. Since $m \mid m_i$, we have $l_i=p(z_i)=q(z_i) \mod m$. Let  $l_i=mx_i+p_i$, where $0 \leq p_i \leq m-1$; then $p_i=p(z_i)=q(z_i) \mod m$.
Slide the edge $e_i'$ over $f_1$ $x_i$ times, so that the obtained edge $f_i$ has $A(f_i)=\Omega(f_i)=n^{p_i}$; the edge $f_1$ is not changed. Thus, we obtained a bouquet $B_K'$ of circles $(f_1, f_2, \ldots, f_k)$, whose GBS group is $K$.

We claim that $B_K'$ satisfies the assertions of the lemma. Indeed, \eqref{slideit1} and \eqref{slideit2} hold by construction. Claim \eqref{slideit3} follows immediately from Lemma \ref{collapse}\eqref{collapseit4}, since $i_2 \neq i_1=1$ satisfies $p_{i_2}=q(z_{i_2})=0$.  Similarly, claim \eqref{slideit4} follows from Lemma \ref{collapse}\eqref{collapseit5}, since for $K \neq G^d_{1,n}$ and $m \geq 2$ we have $p_j=q(z_j)=1$. This proves the lemma.
\end{proof}

\section{Isomorphism criterion for some GBS groups}  \label{sec:iso}

In this section, we show that the description of finite index subgroups provided in Lemma \ref{slide} is in fact a normal form. More generally, we introduce a class of GBS groups and solve the isomorphism problem for this class.

For every $m \geq 1, n, r,s \geq 2, \: 0 \leq p_1, p_2, \ldots, p_{s-1} \leq m-1, \: i=2, \ldots, s$, such that $n=r^l$ for some $l \geq 1$, denote by $\Gamma(n,r;m;p_1,p_2,\ldots,p_{s-1})$ the following GBS graph of groups: it is a bouquet of circles $e_1,e_2, \ldots,e_s$, with $A(e_1)=1$, $\Omega(e_1)=n^m$, $A(e_i)=\Omega(e_i)=n^{p_{i-1}}$, $i=2, \ldots, s$. 

To each graph $\Gamma$ of this form we associate a vector $V(\Gamma)=(v_0,v_1,\ldots,v_{lm-1})$ in $\mathbb{Z}^{lm}$ as follows: $v_j=0$ if $l \nmid j$, and if $j=lj_0$, $0 \leq j_0 \leq m-1$, then $v_j$ is equal to the number of edges $e_i$ among $e_2, \ldots, e_s$ for which we have $p_{i-1}=j_0$. In other words, for every $0 \leq j \leq lm-1$ the number $v_j$ is equal to the number of edges $e_i$ among $e_2, \ldots, e_s$ with labels $A(e_i)=\Omega(e_i)=r^j$.

We say that two vectors $(v_0,v_1,\ldots,v_N)$ and $(w_0,w_1, \ldots, w_N)$ in $\mathbb Z ^ {N+1}$ are cyclic permutations of each other if there exists $C \in \{0,\ldots,N \} $ such that $v_{i'}=w_i$,  where $i'=i+C $ if $i+C \leq N$ and $i'=i+C-N$ otherwise.

Note that if $n_1$ and $n_2$ are not powers of the same number, then the GBS groups defined by the graphs $\Gamma(n_1,r_1;m_1;p_1,p_2,\ldots,p_{k_1-1})$ and
$\Gamma(n_2,r_2;m_2;q_1,q_2,\ldots,q_{k_2-1})$ cannot be isomorphic, since they have different images under the modular homomorphism, see Remark \ref{r4}; therefore, the interesting case is when $r_1=r_2$.
The following theorem tells us when two GBS groups of such form are isomorphic.
Recall that the isomorphism problem for the GBS groups in general is not known to be decidable yet.

\begin{theorem}\label{iso}
    Suppose $n_1=r^{l_1}$ and $n_2=r^{l_2}$, where $r \geq 2$, $l_1, l_2 \geq 1$. 
    Suppose $G_1$ is the GBS group defined by a graph $\Gamma_1=\Gamma(n_1,r;m_1;p_1,p_2,\ldots,p_{k_1-1})$, and 
    $G_2$ is the GBS group defined by a graph $\Gamma_2=\Gamma(n_2,r;m_2;q_1,q_2,\ldots,q_{k_2-1})$, for some $m_1,m_2 \geq 1, k_1,k_2 \geq 2, \; 0 \leq p_i \leq m_1-1, \: i=1, \ldots, k_1-1, \;  0 \leq q_i \leq m_2-1, \: i=1, \ldots, k_2-1$. \\ Then $G_1$ and $G_2$ are isomorphic if and only if the following three conditions hold: 
    \begin{enumerate}
        \item \label{isoit1} $n_1^{m_1}=n_2^{m_2}$. In other words, $l_1m_1=l_2m_2$.
        \item \label{isoit2} $k_1=k_2$.
        \item \label{isoit3} $V(\Gamma_1)$ is a cyclic permutation of $V(\Gamma_2)$.
    \end{enumerate}
\end{theorem}
\begin{proof}
    Recall that two reduced GBS graphs give rise to isomorphic GBS groups if an only if one can get from one to the other by successively applying the slide moves, induction moves and $\mathcal{A}^{\pm 1}$-moves, and so that all the intermediate graphs are reduced, see \cite{CF1}.
    
    We have that $\Gamma_1$ is a bouquet of circles $e_1, \ldots, e_{k_1}$, and $\Gamma_2$ is a bouquet of circles $f_1, \ldots, f_{k_2}$, with $A(e_1)=A(f_1)=1$, $\Omega(e_1)=n_1^{m_1}$, $\Omega(f_1)=n_2^{m_2}$, $A(e_i)=\Omega(e_i)=n^{p_{i-1}}$, $2 \leq i \leq k_1$, and $A(f_i)=\Omega(f_i)=n^{q_{i-1}}$, $2 \leq i \leq k_2$.

    Suppose first that conditions \eqref{isoit1}, \eqref{isoit2}, \eqref{isoit3} hold. Then $e_1$ and $f_1$ already have equal $A$-labels and equal $\Omega$-labels, and it remains to adjust the other edges. Let $k=k_1=k_2$, $S=l_1m_1=l_2m_2$. 
    Let $C$ be such that the cyclic permutation of $V(\Gamma_1)$ by $C$ gives $V(\Gamma_2)$. It's easy to see that it is equivalent to $l_1p_i+C = l_2q_{\sigma(i)} \mod S$, for some permutation $\sigma$ of $\{1, \ldots, k-1 \}$ and for all $i=1, \ldots, k-1$. Then apply the induction move to $\Gamma_1$ $C$ times, each time multiplying the labels of the edges $e_2, \ldots, e_k$ by $r$. We get from $\Gamma_1$ a new bouquet $\Gamma_1'$ of $k$ circles $e_1,e_2', \ldots, e_k'$, with $A(e_1)=1$, $\Omega(e_1)=r^S$, and $A(e_i')=\Omega(e_i')= r^{l_1p_{i-1}+C},$ for $i=2, \ldots, k$. Since $l_1p_{i-1}+C = l_2q_{\sigma(i-1)} \mod S$, we can apply slide moves to $e_2', \ldots, e_k'$ over $e_1$, if necessary, to obtain the graph $\Gamma_2$. This shows that the conditions \eqref{isoit1},\eqref{isoit2},\eqref{isoit3} are sufficient.

    Suppose now that $G_1$ and $G_2$ are isomorphic. It is immediate that the condition \eqref{isoit1} holds, since the images under the modular homomorphism of $G_1$ and $G_2$ must be the same, and $M(G_1)=\langle (n_1)^{m_1} \rangle _{\mathbb{Q}^*}$, $M(G_2)=\langle (n_2)^{m_2} \rangle _{\mathbb{Q}^*}$, by Remark \ref{rem} (since all the loops apart from $e_1$, $f_1$ do not contribute in $M(G_1)$, $M(G_2)$ respectively). Moreover, condition \eqref{isoit2} holds, since $k_i$ is the rank of the quotient of $G_i$ by the subgroup of all the elliptic elements, for $i=1,2$, see Lemma \ref{l3}, so $k_1=k_2$. Denote $k=k_1=k_2$. Therefore, it remains to prove \eqref{isoit3}. If $\Gamma_1$ and $\Gamma_2$ are related only by slide moves and induction moves, this is not hard to see, however since some $\mathcal{A}^{\pm 1}$-moves might be involved, the intermediate graphs might be not bouquets of circles, and so the proof is more complicated.

    Let $\Delta_1, \Delta_2, \ldots, \Delta_N$ be a sequence of reduced GBS graphs, such that $\Delta_1=\Gamma_1$, $\Delta_N=\Gamma_2$ and $\Delta_{i+1}$ is obtained from $\Delta_i$ by a slide move, induction move or $\mathcal{A}^{\pm 1}$-move $M_i$, for each $i=1, \ldots, N-1$. Let $\Delta=\Delta_i$ for some $i=1, \ldots, N$. Note that $\Delta$ defines the GBS group isomorphic to $G_1 \cong G_2$.
    
    \begin{lemma}\label{looplabel}
    For every loop $e$ in $\Delta$  we have $\Omega(e)/A(e)={n_1}^{m_1C}$ for some $C \in \mathbb{Z}$. In particular, if $A(e)=1$, then $\Omega(e)={n_1}^{m_1C}$ for some $C \geq 0$. 
    \end{lemma}
    \begin{proof}
        The claim follows immediately from  the fact that the image under modular homomorphism of $G_1$ consists of the powers of $n_1^{m_1}$, see Remark \ref{rem}.
    \end{proof}
    
   It also follows that we can always suppose that the labels of loops in $\Delta$ are positive (by multiplying both labels of a loop by $-1$ if necessary).

Abusing the notation, below we sometimes denote the vertex in $\Delta_i$ and its image under $M_i$ in $\Delta_{i+1}$ by the same letter, and similar for edges which remain unaffected by the move $M_i$, together with their labels.
    
\begin{lemma}\label{loops}
    In the above notation, at every vertex $v$ of $\Delta$ there exists a loop $e$ such that $A(e)=1$ and $\Omega(e)=n_1^{m_1C}$ for some $C > 0$. 
\end{lemma}
    \begin{proof}
   By Lemma \ref{looplabel}, it suffices to show that at every vertex $v$ of $\Delta$ there exists a strictly ascending loop $e$, i.e. a loop with one label equal to $1$ and the other  label greater than $1$. For $\Delta=\Gamma_1$ such a loop exists: $e=e_1$. Suppose by induction that such a loop exists in $\Delta_i$ and prove that it exists in $\Delta_{i+1}$, $i=1, \ldots, N-1$. Recall that $M_i$ is the move which transforms $\Delta_i$ into $\Delta_{i+1}$.
     
     Indeed, suppose first $M_i$ is a slide move of an edge $E$ over a loop $F$, which changes only the $\Omega$-label of $E$, with $\omega(E)=\alpha(F)=\omega(F)=v$. The strictly ascending loops at vertices not equal to $v$ are not affected. If $F$ or any other loop at $v$ not equal to $E$ is strictly ascending in $\Delta_i$, then its image is also strictly ascending in $\Delta_{i+1}$. So we can suppose that $E$ is a strictly ascending loop, and $F$ is not. If $A(E)=1$ then it stays equal to $1$ after applying $M_i$, and  $\Omega(E)$ stays bigger than $1$, since $F$ is not strictly ascending, therefore, the loop $E$ stays strictly ascending. If $\Omega(E)=1$, then one of the labels of $F$ has to be equal to $1$; since $F$ is not strictly ascending, this implies $A(F)=\Omega(F)=1$, so the slide move is the identity (i.e. does not change the graph nor the labels) and the result follows.
     
     Suppose now $M_i$ is a slide move of an edge $E$ over a non-loop edge $F$, with $\omega(E)=\alpha(F)=u$. It is immediate that every vertex except $u$ will keep the ascending loop, since all the edges incident to these vertices remain unchanged. Since the graph $\Delta_{i+1}$ has to be reduced, it follows that $E$ is not an ascending loop, so there has to be an ascending loop at $u$ which remains unchanged under $M_i$, so in this case the result also follows.
     
     Suppose $M_i$ is an induction move at a vertex $v$, with $E$ being a loop over which the induction is performed. Note that one of the labels of $E$ is equal to $1$, so either $A(E)=\Omega(E)=1$, in which case the induction move is the identity, or $E$ is a strictly ascending loop, which remains unchanged under $M_i$, and the result follows.
     
     Finally, suppose $M_i$ is an $\mathcal{A}^{\pm 1}$-move. Suppose first $M_i$ is an $\mathcal{A}^{-1}$-move, which collapses an edge $E$, such that the vertex $\alpha(E)=u$ has degree 3 and has a strictly ascending loop at it, to a loop at $\omega(E)$. In this case the ascending loops at all the vertices except $u$ are not affected, and the result follows. If $M_i$ is an inverse move, i.e. an $\mathcal{A}$-move, then the new vertex $u$ of degree 3 will have a strictly ascending loop by definition, and the strictly ascending loops at all the other vertices are not affected. This proves the lemma.
    \end{proof}  
        
   Note that although the labels of $\Delta$ don't have to be powers of $n_1$, in general, the proportions between the labels at any vertex always are, at least after we throw out one  strictly ascending loop at each vertex, as the following lemma shows.
        
\begin{lemma}\label{proportions}    
       In the above notation, for every vertex $v$ of $\Delta$, there exists a strictly ascending loop $f_v$ at $v$, such that for every two edges $e,e'$ of $\Delta$ which both begin in $v$ and none of which is equal to $f_v$ or $f_v^{-1}$, the proportion $A(e)/A(e')$ is a power of $n_1$. 
\end{lemma}   
Note that there exists a strictly ascending loop at each vertex of $\Delta$, by Lemma \ref{loops}. Moreover, if there is more than one strictly ascending loop at $v$, then, since the labels of such loops are powers of $n_1$ by Lemma \ref{looplabel}, we will get that in fact for every two edges $e,e'$ of $\Delta$ which both begin in $v$ $A(e)/A(e')=n_1^D$ for some $D \in \mathbb{Z}$. Therefore, the choice of the strictly ascending loop $f_v$ at a vertex $v$ is not essential. \begin{proof}   
  For $\Delta=\Gamma_1$ the claim obviously holds. Suppose by induction that the claim holds in $\Delta_i$ and prove that it holds in $\Delta_{i+1}$, $i=1, \ldots, N-1$, so $e, e' \in \Delta_{i+1}$. Recall that $M_i$ is the move which transforms $\Delta_i$ into $\Delta_{i+1}$. By the above remark, we can always suppose that the claim holds for every choice of strictly ascending loops at each vertex in $\Delta_i$, and it suffices to prove it for some choice of strictly ascending loops at each vertex in $\Delta_{i+1}$.
       
      Suppose first $M_i$ is a slide move of an edge $E$ over a loop $F$, which changes only the $A$-label of $E$, multiplying it by $\Omega(F)/A(F)$, with $\alpha(E)=\alpha(F)=\omega(F)=v$. Let $E'$ be the image of $E$, with $A(E')=A(E)\Omega(F)/A(F)$. Note that the move $M_i$ does not affect any vertices other than $v$, so for these vertices of $\Delta_{i+1}$ the claim holds, and we only need to check it holds at $v$.
      
      Suppose $F$ is a strictly ascending loop. Then we can choose $F$ as $f_v$ both in $\Delta_i$ and $\Delta_{i+1}$. We have $e,e' \neq F$ in $\Delta_{i+1}$. If $e,e' \neq E'$, then they also appear in $\Delta_i$ with the same proportion, so the claim holds. So we can assume that $e'=E'$. In this case we have $$A(e)/A(e')=A(e)/A(E')=A(e)/A(E) \cdot A(F)/\Omega(F),$$
      and $A(e)/A(E)$, $A(F)/\Omega(F)$ are powers of $n_1$, since $e, E \neq F$ and by Lemma \ref{looplabel}. Therefore, $A(e)/A(e')$ is a power of $n_1$, as desired.
      
      Suppose now that $F$ is not a strictly ascending loop. If $E$ is a strictly ascending loop, then $E'$ also has to be a strictly ascending loop, by the definition of a slide move, and we can choose $E$ as $f_v$ in $\Delta_i$ and $E'$ as $f_v$ in $\Delta_{i+1}$. Then we have $e,e' \neq E'$ in $\Delta_{i+1}$, and so they also appear in $\Delta_i$ with the same proportion, and the claim holds.
      Suppose now that both $E$ and $F$ are not strictly ascending loops. Let $f$ be some strictly ascending loop in $\Delta_i$ and $\Delta_{i+1}$, and choose $f_v$ to be $f$ in $\Delta_i$ and $\Delta_{i+1}$.
      Then we have $e,e' \neq f$. If $e, e' \neq E'$, then they also appear in $\Delta_i$ with the same proportion, and the claim holds. So suppose $e'=E'$. As above, we have $A(e)/A(e')=A(e)/A(E')=A(e)/A(E) \cdot A(F)/\Omega(F)$ and $A(e)/A(E)$, $A(F)/\Omega(F)$ are powers of $n_1$, since $e, E \neq f$ and by Lemma \ref{looplabel}. Therefore, $A(e)/A(e')$ is a power of $n_1$, as desired. 
      
    Suppose now that $M_i$ is a slide move of an edge $E$ over a non-loop edge $F$, which changes only the $A$-label of $E$, multiplying it by $\Omega(F)/A(F)$, where $\alpha(E)=\alpha(F)=u$ and $\omega(F)=v$.  Let $E'$ be the image of $E$, with $A(E')=A(E)\Omega(F)/A(F)$. Since the graphs $\Delta_i$, $\Delta_{i+1}$ are reduced, $E$ and $E'$ are not ascending loops. Choose $f_u$, $f_v$ to be some strictly ascending loops at $u$, $v$ respectively, both in $\Delta_i$ and $\Delta_{i+1}$, then $e,e' \neq f_v$.
        If $e, e' \neq E'$, then they also appear in $\Delta_i$ with the same proportion, and the claim holds.  Therefore, we can assume that $e'=E'$, and $e$ begins in $v$. We have 
        $$
        A(e')/A(e)=(A(E)\Omega(F))/(A(F)A(e)) = A(E) / A(F) \cdot A(F^{-1})/A(e).
        $$ 
        Since  $E, F \neq f_u$, and $e, F^{-1} \neq f_v$, we have by induction hypotheses that $A(E) / A(F)=n_1^{D_1}, \:  A(F^{-1})/A(e)=n_1^{D_2}$ for some $D_1, D_2 \in \mathbb{Z}$, and so $A(e')/A(e)=n_1^{D_1+D_2}$, as required.
        
       Suppose $M_i$ is an induction move at a vertex $v$, over a loop $E$. We can suppose that $E$ is strictly ascending, since otherwise the move is the identity. Choose $E$ as the designated loop $f_v$, both in $\Delta_i$ and $\Delta_{i+1}$. The induction move does not change the proportions of the $A$-labels of two edges with the same initial vertex which are both distinct from $E$ and $E^{-1}$, and so the step of induction follows immediately in this case.
       
         Finally, suppose $M_i$ is an $\A$-move. Suppose first $M_i$ is an $\mathcal{A}$-move, which collapses an edge $E$, such that the vertex $\alpha(E)=u$ has degree 3 and has a strictly ascending loop $F$ at it, to a loop $F'$ at $\omega(E)=w$. Note that $A(F')=\Omega(E)$, and $\Omega(F')/A(F')=n_1^C$ for some $C \in \mathbb{Z}$, by Lemma \ref{looplabel}. Choose any strictly ascending loop $f_v$ at $w$, same in $\Delta_i$ and in $\Delta_{i+1}$.
         If neither $e$ nor $e'$ is equal to $F'$ or $(F')^{-1}$, then the claim holds, since the labels are unchanged by $M_i$, so we can suppose that $e=F'$ or $e=(F')^{-1}$ and $e'$ begins in $w$. In the first case, we have $A(e)/A(e')=A(F')/A(e')=A(E^{-1})/A(e')$, and $e',E \neq f_v$, so $A(E^{-1})/A(e')$ is a power of $n_1$ and the claim holds. In the second case, we have $A(e)/A(e')=\Omega(F')/A(e')=n_1^C A(E^{-1})/A(e')$, so the claim also holds. 
         
         Now suppose $M_i$ is an $\mathcal{A}^{-1}$-move. Then the new degree 3 vertex has only one incident edge which is not a strictly ascending loop, so the claim is vacuous for it. For all the other vertices, it is immediate to see that the required proportions in $\Delta_{i+1}$ are all coming from the same proportions in $\Delta_i$, and we can leave the loops $f_v$ unchanged. This completes the proof of the lemma. \end{proof}   
      
        We are going to construct some invariant of the moves. In order to do it, to each $\Delta$ we associate an auxiliary ``dual'' edge-labelled graph $\D$ constructed as follows. Recall that by Lemma \ref{loops}, at each vertex of $\Delta$ there is a strictly ascending loop. Let $\Delta'$ be a graph obtained from $\Delta$ by deleting one strictly ascending loop $f_v$ at each vertex $v$; if there is more than one such loop at some vertex, then delete any, see Figure \ref{fig:2}. The graph $\Delta'$ is labelled by restriction of labelling from $\Delta$.
        
        Let $\D$ have one vertex $v(e)$ for each non-oriented edge $e$ of $\Delta'$. The edges of $\D$ are coloured in $V$ colours, where $V$ is the number of vertices in $\Delta$; the edges corresponding to a vertex $v$ of $\Delta$ (or equivalently $\Delta'$) we call $v$-edges. For each vertex $v$ of $\Delta'$ let $d(v)$ be the set of edges in $\Delta'$ beginning in $v$, both positive and negative, except loops: each loop at $v$ counts as only one edge in $d(v)$ (we choose an arbitrary orientation of it). Then the set of $v$-edges of $\D$ by definition forms a complete graph on $|d(v)|$ vertices, namely, each pair of vertices in $\D$ corresponding to distinct edges in $d(v)$ is connected by a $v$-edge in $\D$, and this applies to all vertices $v$ of $\Delta'$. (If $v$ has degree 3 in $\Delta$, then it has degree 1 in $\Delta'$ and so there are no $v$-edges in $\D$.) This defines $\D$ as a (coloured) graph; note that it can have multiple edges (at most double and of different colours), but has no loops, see Figure \ref{fig:2}.

        Now we define an edge labelling on the graph $\D$ as follows. We now think of $\D$ as a directed graph, choosing an arbitrary orientation of edges. By Lemma \ref{proportions}, if $e$ and $e'$ are two distinct edges in $d(v)$, then $A(e)/A(e')=n_1^C$, for some $C \in \mathbb{Z}$; let $0 \leq C_0 \leq m_1-1$, $C_0 = C \mod m_1$. In this case we label the oriented $v$-edge of $\D$ going from $v(e')$ to $v(e)$ by $C_0$. We do this for both positive and negative edges of $\D$.

 For an edge $E$ of $\D$ we denote its label by $L(E)$. It's easy to see that  we have $L(E^{-1})=-L(E) \mod m_1$. 
        Note that the labelled graph $\D$ does not depend on the choice of the orientation of loops in $d(v)$, by Lemma \ref{looplabel}. Note also that the labelled graph $\D$ doesn't depend on the choice of the strictly ascending loops in the construction of $\Delta'$, since by Lemma \ref{looplabel} all the labels of strictly ascending loops in $\Delta$ are powers of $n_1^{m_1}$, and so if $e'$, $e''$ are strictly ascending loops and $A(e)/A(e')=n_1^C$, $A(e)/A(e'')=n_1^{C'}$, then $C = C' \mod m_1$. Since $\Delta$ is connected, $\D$ is also connected.
        
        If $P=E_1E_2\ldots E_s$ is a path in $\D$, we let $L(P) = L(E_1)+L(E_2)+ \ldots +L(E_s) \mod m_1$, $0 \leq L(P) \leq m_1-1$, be the label of the path $P$. We let the empty path have label $0$.
        Below all the labels of edges and paths are considered modulo $m_1$, and we omit this often.

        The main property of the graph $\D$ is described in the following lemma, see the example in Figure \ref{fig:2}. Recall that $k$ is the number of loops in $\Gamma_1$ and in $\Gamma_2$, and the $p_i$, $q_i$ are as in the statement of the theorem.

    \begin{lemma}\label{Delta}
           The number of vertices in $\D$ is equal to $k-1$, and there is an enumeration $v_1, v_2, \ldots, v_{k-1}$ of the vertex set of $\D$ such that for every $1 \leq i,j \leq k-1$ and every path $P_{ij}$ from $v_i$ to $v_j$ we have $L(P_{ij})=p_j-p_i$.
    \end{lemma}
    \begin{proof}
        First note that the claim holds for $\Delta=\Gamma_1$.
        Indeed, in this case $\Delta'$ is a bouquet of $k-1$ circles, with labels $n^{p_i}$, $i=1, \ldots, k-1$. Therefore, $\D$ has $k-1$ vertices, there is only one colour and $\D$ is a complete graph on $k-1$ vertices $v_1, \ldots, v_{k-1}$, with edges $e_{ij}$ from $v_i$ to $v_j$ having labels $L(e_{ij})=p_j-p_i$, $1 \leq i \neq j \leq k-1$, and the claim follows immediately.
        Suppose by induction that the claim holds in $\Delta_s$ and prove that it holds in $\Delta_{s+1}$, $s=1, \ldots, N-1$. Recall that $M_s$ is the move which transforms $\Delta_s$ into $\Delta_{s+1}$.
        
       Note that the number of edges in $\Delta'$ is equal to the number of edges in $\Delta$ minus the number of vertices in $\Delta$, which is equal to the rank of (topological) fundamental group of $\Delta$ minus 1, i.e., $k-1$. Alternatively, it is easy to see that $M_i$ does not change the number of edges in $\Delta'$. Therefore, $\D$ has $k-1$ vertices.
       
      Let $v_1, v_2, \ldots, v_{k-1}$ be the enumeration of the vertex set of $\D_s$ such that for every $1 \leq i,j \leq k-1$ and every path $P_{ij}$ from $v_i$ to $v_j$ we have $L(P_{ij})=p_j-p_i$. We want to construct a similar enumeration for $\D_{s+1}$.
      
      Consider the following two (abstract) operations on the labelled graph $\D_s$, which don't change the vertex set of $\D_s$: adding an edge $e_1$ from $v_{i_1}$ to $v_{j_1}$ such that $L(e_1)=p_{j_1}-p_{i_1}$ (and $L(e_1^{-1})=p_{i_1}-p_{j_1}$) for some $1 \leq i_1, j_1 \leq k-1$; deleting an edge $e_2$ from $v_{i_2}$ to $v_{j_2}$ such that the graph remains connected after deleting $e_2$, for some $1 \leq i_2, j_2 \leq k-1$. It is immediate to see that the graph obtained by these operations from $\D_s$ still satisfies the condition that for every $1 \leq i,j \leq k-1$ and every path $P_{ij}$ from $v_i$ to $v_j$ in the new graph we have $L(P_{ij})=p_j-p_i$. Therefore, it suffices to see that $\D_{s+1}$ can be obtained from $\D_s$ by applying a finite number of the above two operations, which we call admissible operations.

      It is easy to see that if $M_s$ is an induction move or an $\A$-move, then $\D_{i+1}=\D_i$ as labelled graphs. Therefore, the induction step holds in this case. 
      
      Suppose now that $M_s$ is a slide move of an edge $E$ over a loop  $F$, which changes only the $A$-label of $E$, multiplying it by $\Omega(F)/A(F)$, where $\alpha(E)=\alpha(F)=\omega(F)=v$. Let $E'$ be the image of $E$. As in the proof of Lemma \ref{proportions}, we can choose the distinguished loop $f_v$ at $v$ to be either $E$ in $\Delta_s$ and $E'$ in $\Delta_{s+1}$, or the same in both graphs. Since $\Omega(F)/A(F)$ is a power of $n_1^{m_1}$ by Lemma \ref{looplabel}, the  proportions defining labels of $\D$ don't change modulo $m_1$ when applying $M_s$. It follows that again $\D_{s+1}=\D_s$ as labelled graphs,  so the induction step holds in this case.
     
      The only remaining case to consider is when $M_s$ is a slide move of an edge $E$ over a non-loop edge $F$, which changes only the $A$-label of $E$, multiplying it by $\Omega(F)/A(F)$, where $\alpha(E)=\alpha(F)=v$ and $\omega(F)=w$.  Let $E'$ be the image of $E$, with $A(E')=A(E)\Omega(F)/A(F)$. Since the graphs $\Delta_s$, $\Delta_{s+1}$ are reduced, $E$ and $E'$ are not ascending loops. Choose $f_v$, $f_w$ to be some strictly ascending loops at $v$, $w$ respectively, both in $\Delta_s$ and $\Delta_{s+1}$. 
      
     Recall that for an edge $z$ of $\Delta'$ we denote by $v(z')$ the corresponding vertex of $\D$.
     Let $E_1, \ldots, E_{t}$ be all the edges of $\Delta_s'$ starting in $v$ which are not equal to $E$, $E^{-1}$ and $F$, and $E'_1, \ldots, E'_{t'}$ be all the edges of $\Delta_{s+1}'$ starting in $w$ which are not equal to $E'$, $E'^{-1}$ and $F^{-1}$, where $t, t' \geq 0$.
     
     The set of vertices of $\D_{s+1}$ is the same as the set of vertices of $\D_s$, except that the vertex $v(E)$ is replaced by the vertex $v(E')$. 
     Note that the only edges of $\D_{s+1}$ which can be different from those in $\D_s$ are the $v$-edges and $w$-edges.  
 
     Suppose first that the edge $E$ is not a loop, and $\omega(E) \neq \omega(F)$, so that $E'$ is also not a loop.
    Then the $v$-edges of $\D_s$ form a complete subgraph on $v(E), v(F), v(E_1), \ldots, v(E_t)$, the $w$-edges of $\D_s$ form a complete subgraph on $v(F), v(E'_1), \ldots, v(E'_{t'})$, the $v$-edges of $\D_{s+1}$ form a complete subgraph on $v(F), v(E_1), \ldots, v(E_t)$, and the  $w$-edges of $\D_{s+1}$ form a complete subgraph on $v(E'), v(F), v(E'_1), \ldots, v(E'_{t'})$. 
     It follows that in this case the graph $\D_{s+1}$ can be obtained from $\D_s$ by renaming $v(E)$ to $v(E')$,  adding the $w$-edges connecting $v(E')$ to  $v(F), v(E'_1), \ldots, v(E'_{t'})$ (with the labels as in $\D_{s+1}$), and then deleting the $v$-edges connecting $v(E')$ to $v(F), v(E_1), \ldots, v(E_t)$ (which come from the edges connecting $v(E)$ to $v(F), v(E_1), \ldots, v(E_t)$ in $\D_s$).
     
     Consider first the $w$-edges which are added to $\D_s$. 
     By definition of sliding and of the labels, the label of the new $w$-edge connecting $v(E')$ to  $v(F)$ (in $\D_{s+1}$) is equal to the label of the $v$-edge of $\D_s$ connecting $v(E)$ to $v(F)$. Moreover, the label of the new $w$-edge connecting $v(E')$ to  $v(E_i')$ (in $\D_{s+1}$) is equal to the sum of the label of the $v$-edge of $\D_s$ connecting $v(E)$ to $v(F)$, and the label of the $w$-edge of $\D_s$ connecting $v(F)$ to $v(E_i')$, for each $i=1,\ldots,t'$. It follows that the operations of adding all these edges are admissible operations.
     
     Now consider the $v$-edges which are deleted from the obtained labelled graph. It is immediate that all the graphs obtained in the process of deleting these edges are connected, since $\D_{s+1}$ is connected, therefore, these are also admissible operations. Thus, in this case $\D_{s+1}$ is obtained from $\D_s$ by applying the admissible operations, as desired.
     
     The other cases are similar. Indeed, suppose that $E$ is a loop, then $E'$ is not a loop. Then the $v$-edges of $\D_s$ form a complete subgraph on $v(E), v(F), v(E_1), \ldots, v(E_t)$, the $w$-edges of $\D_s$ form a complete subgraph on $v(F), v(E'_1), \ldots, v(E'_{t'})$, the $v$-edges of $\D_{s+1}$ form a complete subgraph on $v(E), v(F), v(E_1), \ldots, v(E_t)$, and the  $w$-edges of $\D_{s+1}$ form a complete subgraph on $v(E'), v(F), v(E'_1), \ldots, v(E'_{t'})$.
     It follows that in this case the graph $\D_{s+1}$ can be obtained from $\D_s$ by renaming $v(E)$ to $v(E')$ and adding the $w$-edges connecting $v(E')$ to  $v(F), v(E'_1), \ldots, v(E'_{t'})$ (with labels as in $\D_{s+1}$). As before,  we get that $\D_{s+1}$ is obtained from $\D_s$ by applying the admissible operations, as desired.
     
    Finally, suppose that $\omega(E)=\omega(F)$, so $E'$ is a loop. Then the $v$-edges of $\D_s$ form a complete subgraph on $v(E), v(F), v(E_1), \ldots, v(E_t)$, the $w$-edges of $\D_s$ form a complete subgraph on $v(E), v(F), v(E'_1), \ldots, v(E'_{t'})$, the $v$-edges of $\D_{s+1}$ form a complete subgraph on $v(F), v(E_1), \ldots, v(E_t)$, and the  $w$-edges of $\D_{s+1}$ form a complete subgraph on $v(E'), v(F), v(E'_1), \ldots, v(E'_{t'})$.
     It follows that in this case the graph $\D_{s+1}$ can be obtained from $\D_s$ by deleting the $v$-edges connecting $v(E)$ to $v(F), v(E_1), \ldots, v(E_t)$, and renaming $v(E)$ to $v(E')$. As before, we get that $\D_{s+1}$ is obtained from $\D_s$ by applying the admissible operations, as desired. This proves the lemma. 
      \end{proof}

       \begin{center}
\begin{figure}[ht]
\begin{tikzpicture}
    \node (x) at (0,0) [point];
    \node (y) at (1,0) [point];
    \node (z) at (2,0) [point];
    
    \draw[->] (x) -- (y);
    \draw[->] (z) -- (y);
    \path (y) edge[scale=3,out=120,in=60, loop] (y);
    \path[->] (z) edge[bend left=60] (y);

    \node (x) at (0,3) [point];
    \node (y) at (1,3) [point];
    \node (z) at (2,3) [point];
    
    \draw[->] (x) -- (y);
    \draw[->] (z) -- (y);
    \path (y) edge[scale=3,out=120,in=60, loop] (y);
    \path (z) edge[scale=3,out=80,in=20, loop] (z);
    
    \node (x) at (0,6) [point];
    \node (y) at (1,6) [point];
    \node (z) at (2,6) [point];
    
    \draw[->] (x) -- (y);
    \draw[->] (z) -- (y);
    \path (x) edge[scale=3,out=160,in=100, loop] (x);
    \path (y) edge[scale=3,out=170,in=110, loop] (y);
    \path (z) edge[scale=3,out=80,in=20, loop] (z);
    \path (y) edge[scale=3,out=80,in=20, loop] (y);
    \path (z) edge[scale=3,out=-20,in=-80, loop] (z);

    \node (x) at (6,0) [point];
    \node (y) at (7,1) [point];
    \node (z) at (8,0) [point];
    \node (t) at (7,-1) [point];
    
    \draw[->] (x) -- (y);
    \draw[->] (y) -- (z);
    \draw[->] (z) -- (t);
    \draw[->] (t) -- (x);
    \draw[->] (x) -- (z);
    \draw[->] (y) -- (t);
    \path[->,dashed] (z) edge[bend left=50] (t);

    \node (x) at (6,3) [point];
    \node (y) at (7,4) [point];
    \node (z) at (8,3) [point];
    \node (t) at (7,2) [point];
    
    \draw[->] (x) -- (y);
    \draw[->] (y) -- (z);
    \draw[->, dashed] (z) -- (t);
    \draw[->] (x) -- (z);
    
    \node (x) at (7,6) [point];
  
  \path (x) edge[scale=3,out=360,in=300, loop] (x);
  \path (x) edge[scale=3,out=290,in=230, loop] (x);
  \path (x) edge[scale=3,out=220,in=160, loop] (x);
  \path (x) edge[scale=3,out=150,in=90, loop] (x);
  \path (x) edge[scale=3,out=80,in=20, loop] (x);
  
  \node (d1) at (-1.5,6) {$\Delta_1$};
  \node (d1') at (-1.5,3) {$\Delta_1'$};
  \node (d2') at (-1.5,0) {$\Delta_2'$};
  \node (g1) at (9.5,6) {$\Gamma_1$};
  \node (d1b) at (9.5,3) {$\overline{\Delta}_1$};
  \node (d2b) at (9.5,0) {$\overline{\Delta}_2$};
 
\node(slid) at (4.12,6){\LARGE{$\simeq$}};

   \draw[<->,line width=1.5] (3.25,3) -- (5,3);
   \draw[<->,line width=1.5] (3.25,0) -- (5,0);
   \draw[->,line width=1.5] (1,2.5) -- (1,1.5);
   \node(slid) at (1.75,2){\footnotesize sliding};

 \node (vv) at (0.1,-0.2) {\footnotesize $p$};
 \node (vv) at (0.5,-0.2) {\footnotesize $e_1$};
 \node (vv) at (0.8,-0.2) {\footnotesize $p$};
 
 \node (vv) at (1.4,0.2) {\footnotesize $p^2$};
 \node (vv) at (1.7,0.1) {\footnotesize $e_3$};
 \node (vv) at (2,0.2) {\footnotesize $p^3$};
 
 \node (vv) at (1.2,-0.5) {\footnotesize $p^2$};
 \node (vv) at (1.5,-0.4) {\footnotesize $e_4$};
 \node (vv) at (1.9,-0.3) {\footnotesize $p^3$};
 
 \node (vv) at (0.6,0.6) {\footnotesize $1$};
 \node (vv) at (1.4,0.6) {\footnotesize $p^3$};
 \node (vv) at (1,1.2) {\footnotesize $e_2$};

 \node (vv) at (0.1,2.8) {\footnotesize $p$};
 \node (vv) at (0.5,2.8) {\footnotesize $e_1$};
 \node (vv) at (0.8,2.8) {\footnotesize $p$};
 
 \node (vv) at (1.2,2.8) {\footnotesize $p^2$};
 \node (vv) at (1.5,2.8) {\footnotesize $e_3$};
 \node (vv) at (1.8,3.2) {\footnotesize $p^3$};
 
 \node (vv) at (0.6,3.6) {\footnotesize $1$};
 \node (vv) at (1.4,3.6) {\footnotesize $p^3$};
 \node (vv) at (1,4.2) {\footnotesize $e_2$};
 
 \node (vv) at (2.1,3.7) {\footnotesize $p^3$};
 \node (vv) at (2.7,4) {\footnotesize $e_4$};
 \node (vv) at (2.7,3.2) {\footnotesize $p^3$};
 
 \node (vv) at (0.1,5.8) {\footnotesize $p$};
 \node (vv) at (0.8,5.8) {\footnotesize $p$};
 
 \node (vv) at (1.2,5.8) {\footnotesize $p^2$};
 \node (vv) at (1.8,6.2) {\footnotesize $p^3$};
 
 \node (vv) at (1.1,6.6) {\footnotesize $1$};
 \node (vv) at (1.5,6.4) {\footnotesize $p^3$};

 \node (vv) at (0.2,6.2) {\footnotesize $1$};
 \node (vv) at (0.9,6.7) {\footnotesize $p^3$};
 
 \node (vv) at (2.1,6.7) {\footnotesize $p^3$};
 \node (vv) at (2.7,6.2) {\footnotesize $p^3$};
 
 \node (vv) at (-0.7,6.3) {\footnotesize $1$};
 \node (vv) at (-0.1,6.8) {\footnotesize $p^3$};
 
 \node (vv) at (1.9,5.4) {\footnotesize $p^3$};
 \node (vv) at (2.7,5.7) {\footnotesize $1$};
 
 \node (vv) at (-0.8,6.9) {\footnotesize $l_1$};
 \node (vv) at (0.4,6.9) {\footnotesize $l_2$};
 \node (vv) at (1.8,6.9) {\footnotesize $l_5$};
 \node (vv) at (2.8,6.9) {\footnotesize $l_4$};
 \node (vv) at (2.7,5) {\footnotesize $l_3$};


 \node (vv) at (6.7,0.2) {\footnotesize $1$};
 
 \node (vv) at (6.4,0.6) {\footnotesize $2$};
 \node (vv) at (6.4,-0.6) {\footnotesize $2$};
 \node (vv) at (7.1,0.4) {\footnotesize $2$};
 \node (vv) at (7.6,0.6) {\footnotesize $2$};
 \node (vv) at (7.6,-0.6) {\footnotesize $0$};
 \node (vv) at (8.1,-0.6) {\footnotesize $0$};

\node (vv) at (6,-0.3) {\footnotesize $v_1$};
\node (vv) at (7.3,1) {\footnotesize $v_2$};
\node (vv) at (8.3,0) {\footnotesize $v_3$};
\node (vv) at (6.7,-1) {\footnotesize $v_4$};

\node (vv) at (7,3.2) {\footnotesize $1$};
\node (vv) at (6.4,3.6) {\footnotesize $2$};
\node (vv) at (7.6,3.6) {\footnotesize $2$};
\node (vv) at (7.6,2.4) {\footnotesize $0$};

\node (vv) at (6.2,2.7) {\footnotesize $v_1=v(e_1)$};
\node (vv) at (7.8,4) {\footnotesize $v_2=v(e_2)$};
\node (vv) at (8.5,2.7) {\footnotesize $v_3=v(e_3)$};
\node (vv) at (7.8,2) {\footnotesize $v_4=v(e_4)$};

\node (vv) at (7,6.9) {\footnotesize $p$};
\node (vv) at (7.2,6.8) {\footnotesize $1$};
\node (vv) at (7.8,6.3) {\footnotesize $1$};
\node (vv) at (7.9,6) {\footnotesize $p^2$};
\node (vv) at (7.5,5.2) {\footnotesize $p^2$};
\node (vv) at (7.3,5.3) {\footnotesize $p^2$};
\node (vv) at (6.5,5.3) {\footnotesize $p^2$};
\node (vv) at (6.4,5.5) {\footnotesize $p^3$};
\node (vv) at (6.1,6.2) {\footnotesize $1$};
\node (vv) at (6.3,6.5) {\footnotesize $p$};

\node (vv) at (6.4,7.1) {\footnotesize $l_1$};
\node (vv) at (7.8,6.9) {\footnotesize $l_2$};
\node (vv) at (8.1,5.4) {\footnotesize $l_3$};
\node (vv) at (6.9,4.8) {\footnotesize $l_4$};
\node (vv) at (5.8,5.8) {\footnotesize $l_5$};

\end{tikzpicture}
\caption{Example of $\Delta_1$,$\Delta_1'$ and $\overline{ \Delta}_1$ and Lemma \ref{Delta}.  The labels of negative edges of $\D$ are omitted.} \label{fig:2}
\end{figure}
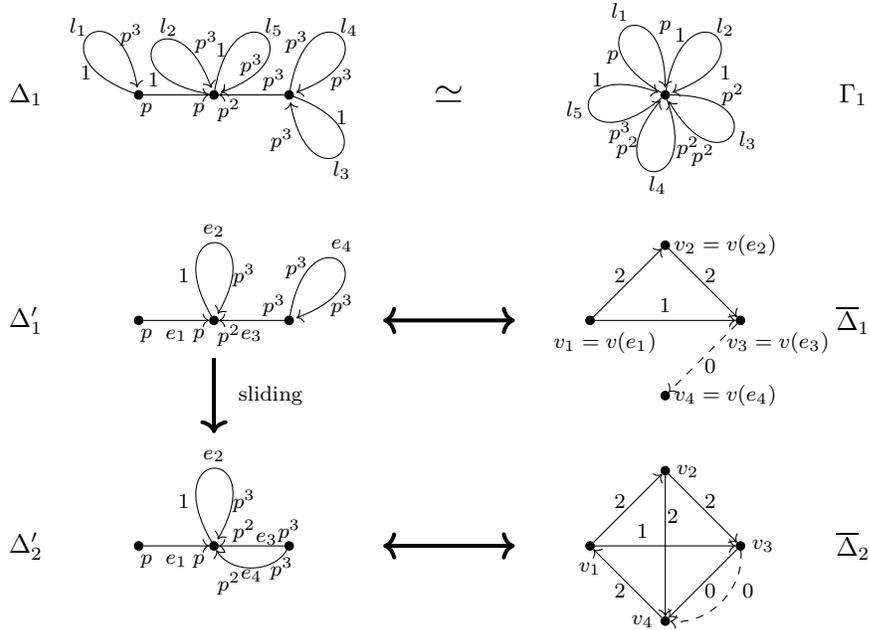
\end{center}

     We now apply Lemma \ref{Delta} to $\Delta=\Delta_N=\Gamma_2$. Since $\Gamma_2$ is a bouquet of circles $f_1, f_2, \ldots, f_k$, with $A(f_1)=1$ and $A(f_2)=n_2^{m_2}$, the graph $\D$ in this case is a complete graph on $k-1$ vertices. We obtain from Lemma \ref{Delta} that there is $\sigma \in S_{k-1}$ such that $A(f_{\sigma(j)+1})/A(f_{\sigma(i)+1})={n_1}^{p_j-p_i+\lambda m_1}=r^{l_1(p_j-p_i+\lambda m_1)}$ for all $1 \leq i < j \leq k-1$ and for some $\lambda \in \mathbb{Z}$. But we also have $A(f_{\sigma(i)+1})={n_2}^{q_{\sigma(i)}}=r^{l_2q_{\sigma(i)}}$, $A(f_{\sigma(j)+1})={n_2}^{q_{\sigma(j)}}=r^{l_2q_{\sigma(j)}}$, so $r^{l_2q_{\sigma(j)}-l_2q_{\sigma(i)}}=r^{l_1p_j-l_1p_i+\lambda l_1 m_1}$, therefore $l_2q_{\sigma(j)}-l_2q_{\sigma(i)}=l_1p_j-l_1p_i \mod S$, where $S=l_1m_1=l_2m_2$. Taking $i=1$, we get $l_2q_{\sigma(j)}-l_2q_{\sigma(1)}=l_1p_j-l_1p_1 \mod S$, so $l_2q_{\sigma(j)}=l_1p_j+C \mod S$  for all $j=1, \ldots, k-1$, where $C=l_2q_{\sigma(1)}-l_1p_1$, i.e., $V(\Gamma_1)$ and $V(\Gamma_2)$ are cyclic permutations of each other, as required. This finishes the proof of Theorem \ref{iso}.
\end{proof}

We can now prove Theorem \ref{isothm}.

{\it Proof of Theorem \ref{isothm}.}
 Given finite generating sets of the subgroups $H_1$, $H_2$, we can algorithmically compute their GBS graphs (corresponding to their induced splittings), by \cite{KWM}. Note that by Remark \ref{rem} if $q_1$ and $q_2$ are not powers of the same number, then the groups $G_{1,q_1}^{d_1}$ and $G_{1,q_2}^{d_2}$ are not commensurable, and so $H_1$ and $H_2$ are not isomorphic. Therefore, we can suppose that $q_1$ and $q_2$ are powers of the same number. It is easy to see that the proofs of Lemma \ref{collapse} and Lemma \ref{slide} are algorithmic, and so for subgroups $H_1$ and $H_2$ we can compute their GBS graphs $\Gamma_1$ and $\Gamma_2$ in the normal form, i.e. as in the statement of Theorem \ref{iso}. Now it follows from Theorem \ref{iso} that we can decide whether $H_1$ and $H_2$ are isomorphic. This proves Theorem \ref{isothm}.

\section{Commensurability of some (G)BS groups}

The next two lemmas are probably well-known and describe the only ways non-solvable Baumslag-Solitar groups can be commensurable.

\begin{lemma}\label{l4}
   Let $n \geq m \geq 1$. Then the groups $BS(m,n)$ and $BS(-m,n)$ are commensurable.
\end{lemma}
\begin{proof}
   Let $G_1=BS(m,n)= \langle a, t \mid t^{-1}a^mt=a^n \rangle$ and $\varphi: G \rightarrow \langle z \rangle _2, \: \psi(a)=1, \: \psi(t)=z$. Then $Ker (\psi)=H_1$ is an index 2 subgroup in $G$, and it is easy to see that $H_1$ is isomorphic to the GBS group given by a simple cycle of length 2, $e_1e_2$, with $A(e_1)=A(e_2)=m, \: \Omega(e_1)=\Omega(e_2)=n$. In the same way, we define $H_2$ of index 2 in $G_2=BS(-m,n)$ isomorphic to the GBS group given by a simple cycle of length 2, $f_1f_2$, with $A(f_1)=A(f_2)=-m, \: \Omega(f_1)=\Omega(f_2)=n$. Now these two GBS groups are isomorphic, which can be seen by changing the signs first of $A(f_1)$ and $\Omega(f_1)$, and then of $\Omega(f_1)$ and $A(f_2)$.
\end{proof}

\begin{lemma}\label{l5}
   Let $n \geq 1$, $k>1$, $l>1$. Then the groups $BS(k,kn)$ and $BS(l,ln)$ are commensurable.
\end{lemma}
\begin{proof}
Suppose first that $n=1$. It is well-known and easy to see that $BS(k,k)$ has a finite index subgroup isomorphic to $F_{k} \times \mathbb{Z}$, so  $BS(k,k)$ and $BS(l,l)$ are commensurable, for all $k,l > 1$. 

Suppose now that $n>1$. It suffices to show that for each $n>1$ and $k>2$ the groups $BS(2,2n)$ and $BS(k,kn)$ are commensurable. By Lemma \ref{l1}, $BS(d,dn)$ has $G_{1,n}^d$ as a finite index subgroup, which is a GBS group  with underlying graph having one vertex and $d$ loops $E_1, \ldots, E_d$ such that $A(E_i)=p$, $\Omega(E_i)=q$, $i=1, \ldots, d$. It suffices to show that $G^{2}_{1,n}$ has $G^{k}_{1,n}$ as a finite index subgroup, for every $k>2$.

Indeed, for each $k > 2$ consider a finite index subgroup $H_k$ of $F_2=\langle a, b \mid \: \rangle$ given by the following covering graph $\Gamma_k$. It has vertices $v_1$, \ldots, $v_{k-1}$, and the following edges: if $k$ is odd, it has loops labelled by $a$ at $v_1$ and $v_{k-1}$, oriented edges labelled by $a$ from $v_{2i}$ to $v_{2i+1}$ and back, $i=1, \ldots, (k-3)/2$ (none if $k=3$), and oriented edges labelled by $b$ from $v_{2i-1}$ to $v_{2i}$ and back, $i=1, \ldots,(k-1)/2$; and if $k$ is even, it has a loop labelled by $a$ at $v_1$, a loop labelled by $b$ at $v_{k-1}$, oriented edges labelled by $a$ from $v_{2i}$ to $v_{2i+1}$ and back, $i=1, \ldots, k/2-1$, and oriented edges labelled by $b$ from $v_{2i-1}$ to $v_{2i}$ and back, $i=1, \ldots, k/2-1$. 

Note that $H_k$ has index $k-1$ in $F_2$, and has rank $k$. Consider the epimorphism $\varphi: G^2_{1,n} \rightarrow F_2$, which has the set of all elliptic elements of $G^2_{1,n}$ as the kernel, and let $H'_k$ be the full preimage of $H_k$ in $G^2_{1,n}$ under $\varphi$. Then $H'_k$ has finite index in $G^2_{1,n}$, and it is easy to see that it has the following GBS structure: the underlying graph is $\Gamma_k$ (with labelling by $a$ and $b$ removed), and for every oriented edge $e$ we have $A(e)=1$, $\Omega(e)=n$.

Choose any maximal subtree in $\Gamma_k$ which consists of edges oriented from $v_i$ to $v_{i+1}$, $i=1, \ldots, k-2$, and apply the collapse moves to these edges for $H_k'$, in any order. It is easy to see that after performing these $k-2$ reductions we obtain a bouquet of circles $S_k$, with edges $e_1, \ldots, e_k$, where for each $i=1,\ldots,k$ we have $A(e_i)=n^{p_i}$, $\Omega(e_i)=n^{q_i}$ for some non-negative integers $p_i, q_i$. Without loss of generality we can suppose that $e_1$ is the edge coming from the loop at $v_{i+1}$, then its label was not changed while collapsing the edges, so $A(e_1)=1$, $\Omega(e_1)=n$. 

Now successively applying slide moves over the edge $e_1$ one can make the edges $e_2, \ldots, e_k$ have labels $A(e_i)=1$ and $\Omega(e_i)=n$, $i=2,\ldots, k$, since these labels were powers of $n$ and sliding over $e_1$ or its inverse allows to divide or multiply a given label of a given edge distinct from $e_1$ by $n$. Since collapse moves and slide moves don't change the GBS group, we obtain that $H_k'$ is isomorphic to $G^k_{1,n}$, and so $G^k_{1,n}$ sits as a finite index subgroup in $G^2_{1,n}$, for every $k > 2$. This completes the proof of the lemma.
\end{proof}

\section{Proof of the main result}

We now complete the proof of Theorem \ref{thm1}.
First we show that $G_1$ and $G_2$ are indeed commensurable in the cases described in the theorem. Indeed, by Lemma \ref{l4} groups $BS(m,n)$ and $BS(-m,n)$ are commensurable. Moreover, as mentioned above, it is easy to see that $BS(1,n^k)$ embeds as a finite index subgroup in $BS(1,n)$ for $n,k \geq 1$, so $BS(1,n^k)$ and $BS(1,n^l)$ are commensurable for all $n,k,l \geq 1$ (see \cite{FM}). So if $|m_1|=|m_2|=1$ and $n_1, n_2$ are powers of the same integer, then $BS(m_1,n_1)$ and $BS(m_2,n_2)$ are commensurable. Finally, suppose that $|m_1|>1$, $|m_2| >1$, $m_1  \mid  n_1$, $m_2 \mid n_2$ and $n_1/|m_1|=n_2/|m_2|$. By Lemma \ref{l5} the groups $BS(|m_1|,n_1)$ and $BS(|m_2|,n_2)$ are commensurable, and, therefore, by Lemma \ref{l4}, $BS(m_1,n_1)$ and $BS(m_2,n_2)$ are also commensurable, as desired.

To prove the converse, we will need the following two lemmas.

\begin{lemma}\label{non-ascending}
	Suppose that $n_1 > m_1 > 1$ and $n_2 > m_2 > 1$,  $m_1 \nmid n_1$, $m_2 \nmid n_2$, and the pairs $(m_1,n_1)$ and $(m_2,n_2)$ are distinct. Then the groups $BS(m_1,n_1)$ and $BS(m_2, n_2)$ are not commensurable.
\end{lemma}
\begin{proof}
We use the notation of previous sections. Let $d_1=\gcd(m_1,n_1)$ and $d_2=\gcd(m_2,n_2)$, $m_1=p_1d_1, \: n_1=q_1d_1, \: m_2=p_2d_2, \: n_2=q_2d_2$, and $q_1>p_1>1$, $q_2>p_2>1$. Since $G^{d_i}_{p_i,q_i}$ embeds as a finite index subgroup in $BS(m_i,n_i)$ for $i=1,2$ by Lemma \ref{l1}, it suffices to prove that $G_1=G^{d_1}_{p_1,q_1}$ and $G_2=G^{d_2}_{p_2,q_2}$ are not commensurable.

Suppose $K_1$ is a finite index subgroup of $G_1$ and $K_2$ is a finite index subgroup of $G_2$, then $K_1$ and $K_2$ satisfy the conclusions of Lemma \ref{l3}. Let $\Gamma_{K_i}$ be the covering graphs which are GBS graphs for $K_i$, $i=1,2$, as in Lemma \ref{l3}. Then $\Gamma_{K_i}$ is reduced, and the degree of each vertex in $K_i$ is equal to $2d_i$, for $i=1,2$. It follows that if $V_i$ is the number of vertices in $\Gamma_{K_i}$, then the number of edges in $\Gamma_{K_i}$ is $d_iV_i$, and so the rank of the (topological) fundamental group of the underlying graph of $\Gamma_{K_i}$ is equal to $V_i (d_i-1)$, $i=1,2$.

Note that since $m_i \nmid n_i$, the group $BS(m_i,n_i)$ has no non-trivial integral moduli, and so also the GBS groups $G_i$ and $K_i$ have no non-trivial integral moduli, by Remark \ref{r4}. This means that the deformation spaces for $K_1$ and $K_2$ are non-ascending, since the existence of a strictly ascending loop immediately implies existence of a non-trivial integral modulus, by Remark \ref{rem}. It follows that any two graphs in such a deformation space are related by a finite sequence of slide moves (see \cite{F2} and \cite{GL}). 

Suppose that $K_1$ and $K_2$ are isomorphic, then $\Gamma_{K_1}$ and $\Gamma_{K_2}$ are in the same deformation space. 

Suppose first $p_1=q_1$ and $p_2=q_2$, then $d_1 \neq d_2$. Since $\Gamma_{K_1}$ and $\Gamma_{K_2}$ are related by slide moves, they have the same number of vertices: $V_1=V_2$. But also $V_1(d_1-1)=V_2(d_2-1)$, since deformations never change the rank of the (topological) fundamental group, contradiction.

Then we have $p_1 \neq q_1$ or $p_2 \neq q_2$. But the slide moves applied to $\Gamma_{K_1}$ will never change the set of labels of all edges, which is $\{ p_1, q_1 \}$ for $\Gamma_{K_1}$ and $\{ p_2, q_2 \}$ for $\Gamma_{K_2}$, so it is impossible to obtain $\Gamma_{K_2}$ from $\Gamma_{K_1}$ by slide moves, contradiction.
\end{proof}

\begin{lemma}\label{ascending}
	Suppose that $n_1>m_1>1$, $n_2>m_2 >1$, $m_1  \mid  n_1$, $m_2 \mid n_2$, and $\frac{n_1}{m_1} \neq \frac{n_2}{m_2}$. Then the groups $BS(m_1,n_1)$ and $BS(m_2, n_2)$ are not commensurable.
\end{lemma}

\begin{proof}
Let $n_1=m_1d_1$, $n_2=m_2d_2$, then $d_1 \neq d_2$. It follows from Remark \ref{r4} that if $d_1$ and $d_2$ do not have a common power, then $BS(m_1,n_1)$ and $BS(m_2,n_2)$ are not commensurable. Therefore, we can suppose that there exist $r,l_1,l_2$ such $d_1=r^{l_1}$, $d_2=r^{l_2}$. We need to show that $BS(m_1, m_1r^{l_1})$ and $BS(m_2, m_2r^{l_2})$ are not commensurable, when $l_1 \neq l_2$. Note that by Lemma \ref{l5}, $BS(m_1,m_1r^{l_1})$ is commensurable to $BS(2,2r^{l_1})$ and $BS(m_2,m_2r^{l_2})$ is commensurable to $BS(2,2r^{l_2})$.
Moreover, by Lemma \ref{l1}, $BS(2,2r^{l_1})$ is commensurable to $G^2_{1,r^{l_1}}$ and  $BS(2,2r^{l_2})$ is commensurable to $G^2_{1,r^{l_2}}$.
Therefore, it suffices to show that $G_1=G^2_{1,r^{l_1}}$ and $G_2=G^2_{1,r^{l_2}}$ are not commensurable, for any $r>1$, $l_1 \neq l_2$, $l_1,l_2 \geq 1$.

Suppose on the contrary that $G_1$ and $G_2$ are commensurable, and $K_1 \leq G_1$ and $K_2 \leq G_2$ are isomorphic finite index subgroups. Note first that $G_1$ and $G_2$ are not isomorphic, since they have different images under the modular homomorphism (generated by $r^{l_1}$ for $G_1$ and by $r^{l_2}$ for $G_2$), see Remark \ref{r4}. Furthermore, note that $K_1 \neq G_1$. Indeed, the quotient of $G_1$ over all the elliptic elements gives $F_2$, and it is easy to see from Lemma \ref{l3} that for every proper finite index subgroup $K_2 \leq G_2$ the quotient of $K_2$ over all the elliptic elements is a free group of rank at least 3 (since it is a proper finite index subgroup of $F_2$, given by the covering graph  $\Gamma_{K_2}$). Therefore, we can assume that $K_1 \neq G_1$ and $K_2 \neq G_2$.

Let $m_1$ for $K_1$, $m_2$ for $K_2$ be defined as in (\ref{m}).
By Lemma \ref{slide}, $K_1$ is isomorphic to a GBS group defined by a bouquet $B_{K_1}$ of circles $f_1, \ldots, f_k$, such that $A(f_1)=1$, $\Omega(f_1)=r^{l_1m_1}$, and for each $i=2, \ldots, k$ we have $A(f_i)=\Omega(f_i)=r^{l_1p_i}$, where $0 \leq p_i \leq m_1-1$. Similarly, $K_2$ is isomorphic to a GBS group defined by a bouquet $B_{K_2}$ of circles $f_1', \ldots, f_k'$, such that $A(f_1')=1$, $\Omega(f_1')=r^{l_2m_2}$, and for each $i=2, \ldots, k$ we have $A(f_i)=\Omega(f_i)=r^{l_2q_i}$, where $0 \leq q_i \leq m_2-1$. Note that the number of loops is indeed the same, since the quotients over all the elliptic elements must be isomorphic. Moreover, $k \geq 3$, since $K_1 \neq G_1$ and $K_2 \neq G_2$.  

By Theorem \ref{iso}, we get that $l_1m_1=l_2m_2$ and $V(B'_{K_1})$ is a cyclic permutation of $V(B'_{K_2})$. Denote $l_1m_1$ by $L$. Let $V(B'_{K_1})=(v_0,v_1,\ldots,v_{L-1})$ and $V(B'_{K_2})=(w_0,w_1,\ldots,w_{L-1})$.

Since $l_1 \neq l_2$, we have $m_1 \neq m_2$. Without loss of generality, we can assume that $m_1 > m_2$. In particular, $m_1 \geq 2$. Since also $k \geq 3$,  it follows from Lemma \ref{slide} that there exist $i, j \in \{2, \ldots, k\}$ such that $A(e_i)=\Omega(e_i)=1$ and $A(e_j)=\Omega(e_j)=r^{l_1}$. In other words, $v_0 \neq 0$ and $v_{l_1} \neq 0$. However, since $l_1m_1=l_2m_2$ and $m_1>m_2$, we have $l_1<l_2$, so neither the vector $(w_0,w_1,\ldots,w_{L-1})$ nor its cyclic permutations can have two non-zero coordinates with indices differing by $l_1$, since in $(w_0,w_1,\ldots,w_{L-1})$ only the coordinates with indices dividing $l_2$ can be non-zero. But $(w_0,w_1,\ldots,w_{L-1})$ is a cyclic permutation of $(v_0,v_1,\ldots,v_{L-1})$, which has two non-zero coordinates with indices differing by $l_1$. This is a contradiction. 
\end{proof}

Now suppose that none of the conditions (1), (2), (3) of Theorem \ref{thm1} hold. We claim that in this case $G_1$ and $G_2$ are not commensurable. By Farb and Mosher \cite{FM}, this is true in the solvable case, so we can assume that $n_1 \geq |m_1|  > 1$ and $n_2 \geq |m_2| > 1$. Moreover, since $BS(m,n)$ and $BS(-m,n)$ are commensurable, we can assume that $m_1,m_2 > 0$. Note also that if $m_1=n_1$ and $m_2>n_2$, then $BS(m_1, n_1)$ and $BS(m_2,n_2)$ cannot be commensurable, in fact they are not even quasi-isometric \cite{W}. Therefore, we can assume that $n_1 > m_1 > 1$ and $n_2 > m_2 > 1$.

By Remark \ref{r4},  if $m_1 \mid n_1$ and $m_2 \nmid n_2$,  then $BS(m_1,n_1)$ and $BS(m_2,n_2)$ are not commensurable, and the same is true if $m_1 \nmid n_1$ and $m_2 \mid n_2$.
Therefore, there are two cases to consider  -- the non-ascending case, when $m_1 \nmid n_1$, $m_2 \nmid n_2$, and the ascending case, when $m_1 \mid n_1$, $m_2 \mid n_2$. In the non-ascending case Lemma \ref{non-ascending} implies that  $BS(m_1,n_1)$ and $BS(m_2,n_2)$ are not commensurable, and in the ascending case Lemma \ref{ascending} implies that  $BS(m_1,n_1)$ and $BS(m_2,n_2)$ are not commensurable. This completes the proof of Theorem \ref{thm1}.

\end{document}